\pdfoutput=1
\documentclass{amsart}
\usepackage{a4wide}
\usepackage{tikz}
\usepackage{amscd}
\usepackage{amssymb}
\usepackage{graphicx,color}
\graphicspath{{figures/}}

\usepackage{hyperref}

\theoremstyle{plain}
\newtheorem{theorem}{Theorem}
\newtheorem{proposition}{Proposition}
\newtheorem{lemma}{Lemma}
\newtheorem{conjecture}{Conjecture}
\newtheorem{fact}{Fact}
\newtheorem{corollary}{Corollary}
\newtheorem{problem}{Problem}

\theoremstyle{definition}
\newtheorem{definition}{Definition}
\newtheorem{example}{Example}
\newtheorem{notation}{Notation}

\theoremstyle{remark}
\newtheorem{remark}{Remark}

\newcommand{\vone}{\begin{picture}(30,0)
\put(5,2){\circle{10}}
\put(10,2){\vector(1,0){10}}
\put(25,2){\circle{10}}
\end{picture}
}

\newcommand{\voneA}{\begin{picture}(30,0)
\put(5,2){\circle{10}}
\put(20,2){\vector(-1,0){10}}
\put(25,2){\circle{10}}
\end{picture}
}

\newcommand{\atrba}{~
\,\begin{picture}(10,13)
\put(5,3){\circle{15}}
\put(0,-2){\vector(1,1){10}}
\put(10,-2){\vector(-1,1){10}}
\put(5,10){\vector(0,-1){15}}
\put(2,10){\circle*{3}}
\end{picture}~~~
}

\newcommand{\atrbb}{~
\,\begin{picture}(10,13)
\put(5,3){\circle{15}}
\put(0,-2){\vector(1,1){10}}
\put(10,-2){\vector(-1,1){10}}
\put(5,10){\vector(0,-1){15}}
\put(8,10){\circle*{3}}
\end{picture}~~~
}

\newcommand{\atraa}{~
\,\begin{picture}(10,13)
\put(5,3){\circle{15}}
\put(0,-2){\vector(1,1){10}}
\put(10,-2){\vector(-1,1){10}}
\put(5,-5){\vector(0,1){15}}
\put(12,2){\circle*{3}}
\end{picture}~~~
}

\newcommand{\atrac}{~
\,\begin{picture}(10,13)
\put(5,3){\circle{15}}
\put(0,-2){\vector(1,1){10}}
\put(10,-2){\vector(-1,1){10}}
\put(5,-5){\vector(0,1){15}}
\put(2,10){\circle*{3}}
\end{picture}~~~
}
\newcommand{\atrad}{~
\,\begin{picture}(10,13)
\put(5,3){\circle{15}}
\put(0,-2){\vector(1,1){10}}
\put(10,-2){\vector(-1,1){10}}
\put(5,-5){\vector(0,1){15}}
\put(-2,2){\circle*{3}}
\end{picture}~~~
}
\newcommand{\atrae}{~
\,\begin{picture}(10,13)
\put(5,3){\circle{15}}
\put(0,-2){\vector(1,1){10}}
\put(10,-2){\vector(-1,1){10}}
\put(5,-5){\vector(0,1){15}}
\put(2,-4){\circle*{3}}
\end{picture}~~~
}
\newcommand{\atraf}{~
\,\begin{picture}(10,13)
\put(5,3){\circle{15}}
\put(0,-2){\vector(1,1){10}}
\put(10,-2){\vector(-1,1){10}}
\put(5,-5){\vector(0,1){15}}
\put(8,-4){\circle*{3}}
\end{picture}~~~
}

\newcommand{\ahaa}{~
\begin{picture}(10,13)
\put(5,3){\circle{15}}
\put(2,-4){\vector(0,1){14}}
\put(8,10){\vector(0,-1){14}}
\put(-3,3){\vector(1,0){16}}
\put(11,8){\circle*{3}}
\end{picture}~~~
}
\newcommand{\ahaas}{~
\begin{picture}(10,13)
\put(5,3){\circle{15}}
\put(-1,12){$+$}
\put(2,-4){\vector(0,1){14}}
\put(6,-8){$-$}
\put(8,10){\vector(0,-1){14}}
\put(13,1){$+$}
\put(-3,3){\vector(1,0){16}}
\put(11,8){\circle*{3}}
\end{picture}~~~
}

\newcommand{\ahaam}{~
\begin{picture}(10,13)
\put(5,3){\circle{15}}
\put(-1,12){$+$}
\put(2,-4){\vector(0,1){14}}
\put(6,-8){$-$}
\put(8,10){\vector(0,-1){14}}
\put(13,1){$-$}
\put(-3,3){\vector(1,0){16}}
\put(11,8){\circle*{3}}
\end{picture}~~~
}

\newcommand{\ahab}{~
\begin{picture}(10,13)
\put(5,3){\circle{15}}
\put(2,-4){\vector(0,1){14}}
\put(8,10){\vector(0,-1){14}}
\put(-3,3){\vector(1,0){16}}
\put(6,10){\circle*{3}}
\end{picture}~~~
}
\newcommand{\ahac}{~
\begin{picture}(10,13)
\put(5,3){\circle{15}}
\put(2,-4){\vector(0,1){14}}
\put(8,10){\vector(0,-1){14}}
\put(-3,3){\vector(1,0){16}}
\put(-2,6){\circle*{3}}
\end{picture}~~~
}
\newcommand{\ahad}{~
\begin{picture}(10,13)
\put(5,3){\circle{15}}
\put(2,-4){\vector(0,1){14}}
\put(8,10){\vector(0,-1){14}}
\put(-3,3){\vector(1,0){16}}
\put(-1,-1){\circle*{3}}
\end{picture}~~~
}
\newcommand{\ahae}{~
\begin{picture}(10,13)
\put(5,3){\circle{15}}
\put(2,-4){\vector(0,1){14}}
\put(8,10){\vector(0,-1){14}}
\put(-3,3){\vector(1,0){16}}
\put(5,-4.5){\circle*{3}}
\end{picture}~~~
}

\newcommand{\ahaf}{~
\begin{picture}(10,13)
\put(5,3){\circle{15}}
\put(2,-4){\vector(0,1){14}}
\put(8,10){\vector(0,-1){14}}
\put(-3,3){\vector(1,0){16}}
\put(11,-1){\circle*{3}}
\end{picture}~~~
}

\newcommand{\dxup}{~
\begin{picture}(10,15)
\put(5,11.5){\circle*{3}}
\put(5,4){\circle{15}}
\put(0,-1){\vector(1,1){10}}
\put(10,-1){\vector(-1,1){10}}
\put(8,12){+}
\put(-4,12){+}
\end{picture}~
}
\newcommand{\dxump}{~
\begin{picture}(10,15)
\put(5,11.5){\circle*{3}}
\put(5,4){\circle{15}}
\put(0,-1){\vector(1,1){10}}
\put(10,-1){\vector(-1,1){10}}
\put(8,12){+}
\put(-4,12){-}
\end{picture}~
}
\newcommand{\dxupm}{~
\begin{picture}(10,15)
\put(5,11.5){\circle*{3}}
\put(5,4){\circle{15}}
\put(0,-1){\vector(1,1){10}}
\put(10,-1){\vector(-1,1){10}}
\put(10,12){-}
\put(-4,12){+}
\end{picture}~
}

\newcommand{\dxb}{~
\begin{picture}(10,15)
\put(5,4){\circle{15}}
\put(0,-1){\vector(1,1){10}}
\put(10,-1){\vector(-1,1){10}}
\put(5,-3.5){\circle*{3}}
\end{picture}~
}
\newcommand{\dxbp}{~
\begin{picture}(10,15)
\put(5,-3.5){\circle*{3}}
\put(5,4){\circle{15}}
\put(0,-1){\vector(1,1){10}}
\put(10,-1){\vector(-1,1){10}}
\put(8,12){+}
\put(-4,12){+}
\end{picture}~
}

\newcommand{\dxbmp}{~
\begin{picture}(10,15)
\put(5,-3.5){\circle*{3}}
\put(5,4){\circle{15}}
\put(0,-1){\vector(1,1){10}}
\put(10,-1){\vector(-1,1){10}}
\put(8,12){+}
\put(-4,12){-}
\end{picture}~
}
\newcommand{\dxbpm}{~
\begin{picture}(10,15)
\put(5,-3.5){\circle*{3}}
\put(5,4){\circle{15}}
\put(0,-1){\vector(1,1){10}}
\put(10,-1){\vector(-1,1){10}}
\put(10,12){-}
\put(-4,12){+}
\end{picture}~
}
\newcommand{\dxbm}{~
\begin{picture}(10,15)
\put(5,-3.5){\circle*{3}}
\put(5,4){\circle{15}}
\put(0,-1){\vector(1,1){10}}
\put(10,-1){\vector(-1,1){10}}
\put(10,12){-}
\put(-4,12){-}
\end{picture}~
}

\newcommand{\reduced}{\operatorname{reduced}}

\newcommand{\sign}{\operatorname{sign}}

\newcommand{\cNsum}{ \sum_{\check{n}_{b-1} + 1 \le i \le \check{n}_d
}
}

\newcommand{\cat}{\alpha_i \tilde{x}^*_i}
\newcommand{\caFsum}{\sum_{\check{n}_{b-1} + 1 \le i \le \check{n}_{d}} \alpha_i \tilde{x}^*_i }
\newcommand{\caFsums}{\sum_{\check{n}_{1} + 1 \le i \le \check{n}_{3}} \alpha_i \tilde{x}^*_i }

\newcommand{\cFsum}{\sum_{\check{n}_{b-1} + 1 \le i \le \check{n}_{d}} \alpha_i x^*_i}
\newcommand{\cFsums}{\sum_{\check{n}_{1} + 1 \le i \le \check{n}_{3}} \alpha_i x^*_i}
\newcommand{\scFsum}{\sum_{\check{n}_{b-1} + 1 \le i \le \check{n}_{d}} \alpha_i}

\newcommand{\caFsumIrr}{\sum_{i \in \cirri} \alpha_i \tilde{x}^*_i
}
\newcommand{\caFsumIrrs}{\sum_{i \in \cirris} \alpha_i \tilde{x}^*_i
}

\newcommand{\cFsumIrr}{\sum_{i \in \cirri} \alpha_i x^*_i}
\newcommand{\cFsumIrrs}{\sum_{i \in \cirris} \alpha_i x^*_i}

\newcommand{\caFsumConn}{\sum_{i \in \cconni} \alpha_i \tilde{x}^*_i 
}
\newcommand{\caFsumConns}{\sum_{i \in \cconnis} \alpha_i \tilde{x}^*_i 
}
\newcommand{\cFsumConn}{\sum_{i \in \cconni} \alpha_i x^*_i}
\newcommand{\cFsumConns}{\sum_{i \in \cconnis} \alpha_i x^*_i}

\newcommand{\cineq}{\check{n}_{b-1} + 1 \le i \le \check{n}_{d}}
\newcommand{\cR}{\check{R}}

\newcommand{\cyc}{\operatorname{cyc}}
\newcommand{\rev}{\operatorname{rev}}

\newcommand{\cRep}{\check{R}_{\epsilon_1 \epsilon_2 \epsilon_3 \epsilon_4 \epsilon_5}}

\newcommand{\bax}{~
\begin{picture}(10,12)
\put(5,-4.5){\circle*{3}}
\put(5,3){\circle{15}}
\put(0,-2){\vector(1,1){10}}
\put(10,-2){\vector(-1,1){10}}
\end{picture}~~~
}
\newcommand{\uax}{~
\begin{picture}(10,12)
\put(5,10){\circle*{3}}
\put(5,3){\circle{15}}
\put(0,-2){\vector(1,1){10}}
\put(10,-2){\vector(-1,1){10}}
\end{picture}~~~
}

\newcommand{\atrb}{~
\,\begin{picture}(10,13)
\put(5,3){\circle{15}}
\put(0,-2){\vector(1,1){10}}
\put(10,-2){\vector(-1,1){10}}
\put(5,10){\vector(0,-1){15}}
\end{picture}~~~
}

\newcommand{\aha}{~
\begin{picture}(10,13)
\put(5,3){\circle{15}}
\put(2,-4){\vector(0,1){14}}
\put(8,10){\vector(0,-1){14}}
\put(-3,3){\vector(1,0){16}}
\end{picture}~~~
}

\newcommand{\conn}{\operatorname{Conn}}
\newcommand{\irr}{\operatorname{Irr}}

\newcommand{\cirri}{\check{I}^{(\operatorname{Irr})}_{b, d}}
\newcommand{\cirris}{\check{I}^{(\operatorname{Irr})}_{2, 3}}

\newcommand{\cconni}{\check{I}^{(\operatorname{Conn})}_{b, d}}
\newcommand{\cconnis}{\check{I}^{(\operatorname{Conn})}_{2, 3}}
\newcommand{\sub}{\operatorname{Sub}}
\newcommand{\ii}{{\rm{I}}}
\newcommand{\sss}{{\rm{SI\!I\!I}}}
\newcommand{\www}{{\rm{WI\!I\!I}}}
\newcommand{\s}{{\rm{SI\!I}}}
\newcommand{\w}{{\rm{WI\!I}}}

\begin{document}
\title[GPV conjecture for degree three case]{Goussarov-Polyak-Viro Conjecture for degree three case
}
\thanks{\color{black}{MSC2020: 57K16; 57K12; 57K10}}
\author{Noboru Ito}
\address{Graduate School of Mathematical Sciences, The University of Tokyo, 3-8-1, Komaba, Meguro-ku, Tokyo 153-8914, Japan}
\address{(Current) National Institute of Technology, Ibaraki College, 866 Nakane,  Hitachinaka, Ibaraki, 312-8508, Japan}
\email{nito@gm.ibaraki-ct.ac.jp}
\author{Yuka Kotorii}
\address{
Mathematics Program, Graduate school of Advanced Science and Engineering, Hiroshima University
\\
1-7-1 Kagamiyama Higashi-hiroshima City, Hiroshima 739-8521, Japan 
}
\address{Mathematical Analysis Team, RIKEN Center for Advanced Intelligence Project (AIP) 
\\
1-4-1 Nihonbashi, Chuo-ku, Tokyo 103-0027, Japan}
\address{interdisciplinary Theoretical \& Mathematical Sciences Program (iTHEMS) RIKEN
\\
 2-1, Hirosawa, Wako, Saitama 351-0198, Japan}
\email{kotorii@hiroshima-u.ac.jp}
\author{Masashi Takamura}
\address{School of Social Informatics, Aoyama Gakuin University,
5-10-1 Fuchinobe Chuo-ku, Sagamihara-shi, Kanagawa-ken 252-5258, Japan}
\email{takamura@si.aoyama.ac.jp}
\date{August 20, 2023}
\maketitle

\begin{abstract}
Although it is known that the dimension of the Vassiliev invariants of degree three of long virtual knots is seven, the complete list of seven distinct Gauss diagram formulas have been unknown explicitly, where only one known  formula was revised without proof.     
In this paper, we give seven Gauss diagram formulas to present the seven invariants of the degree three    (Proposition~\ref{thm1}).  We further give $23$ Gauss diagram formulas of classical knots  (Proposition~\ref{thm2}). 
In particular, the Polyak-Viro Gauss diagram formula \cite{PV} is not a long virtual knot invariant; however, it is included in the list of $23$ formulas.  It has been unknown whether this formula would be available by arrow diagram calculus automatically.       
In consequence, as it relates to the conjecture of Goussarov-Polyak-Viro \cite[Conjecture~3.C]{gpv}, for all the degree three finite type long virtual knot invariants, each Gauss diagram formula is represented as those  of Vassiliev invariants of classical knots (Theorem~\ref{thm3}).  
\end{abstract}

\section{Introduction}\label{intro}
In 1990, Vassiliev introduced a powerful family of filtered classical knot invariants \cite{V}, called finite type invariants or Vassiliev knot invariants.   Nowadays, since Vassiliev knot invariants have been extended to general or other objects, e.g. braids, (long) virtual knots, or $3$-manifolds.   Throughout of this paper, by ``Vassiliev knot invariants", we mean classical knot invariants of finite type as follows.      
A knot invariant taking values in an abelian group is said to be \emph{finite type of degree at most $n$} if for any singular knot with $n+1$ singular crossings, its value of the linear extended invariant vanishes, where each singular crossing is defined as follows
\begin{equation*}
\begin{picture}(0,20)
\put(-30,-6){\includegraphics[width=4cm]{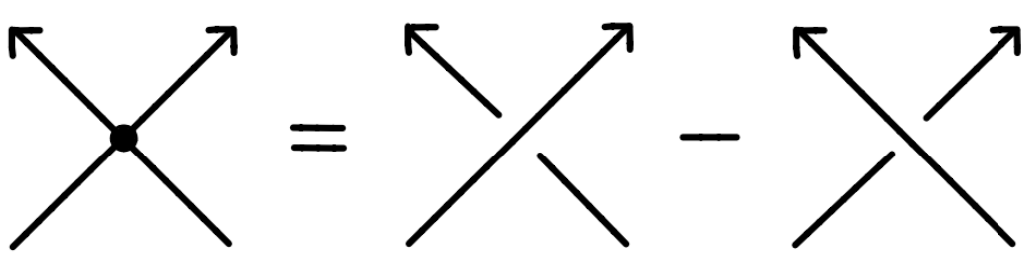}\quad.}
\end{picture}
\end{equation*}

A remarkable relationship between the Vassiliev knot invariants and the Jones polynomial is known: each coefficient of power series obtained from the Jones polynomial by replacing its variable by $e^x$ is a Vassiliev knot invariant \cite[Theorem, Page~227]{BL}.  A famous open problem concerning Vassiliev knot invariants is whether they distinguish nonisotopic knots (the Vassiliev Conjecture \cite[Page~46]{PS}).   

For the theory of Vassiliev knot invariants, in order to present a Vassiliev knot invariant by a combinatorial formula,
Polyak and Viro introduced a Gauss diagram formula in \cite{PV}   
(for the definition of Gauss diagram formulas, see Definition~\ref{remarkGauss}).
Polyak and Viro \cite[Page~450, Section~6]{PV} had a question: ``{\it Can any Vassiliev invariant be calculated as a function of arrow polynomials evaluated on the knot diagram?}"      Here, ``a function of arrow polynomials" (``Vassiliev invariant",~resp.) means a Gauss diagram formula (Vassiliev knot  invariant,~resp.).    
Goussarov gave a positive answer to this question in a framework of a (long) virtual knot theory \cite[Theorem~3.A]{gpv}.  In fact, in  \cite{gpv}, they introduce a definition of finite type invariants;  a (long) virtual knot invariant taking values in an abelian group is said to be \emph{Goussarov-Polyak-Viro finite type of degree at most $n$} if the value of the linear extended invariant vanishes for any (long) semi-virtual knot with $n+1$ semi-virtual crossings, where each semi-virtual crossing is defined as follows \cite{gpv}.
\begin{equation*}
\begin{picture}(0,20)
\put(-30,-6){\includegraphics[width=4cm]{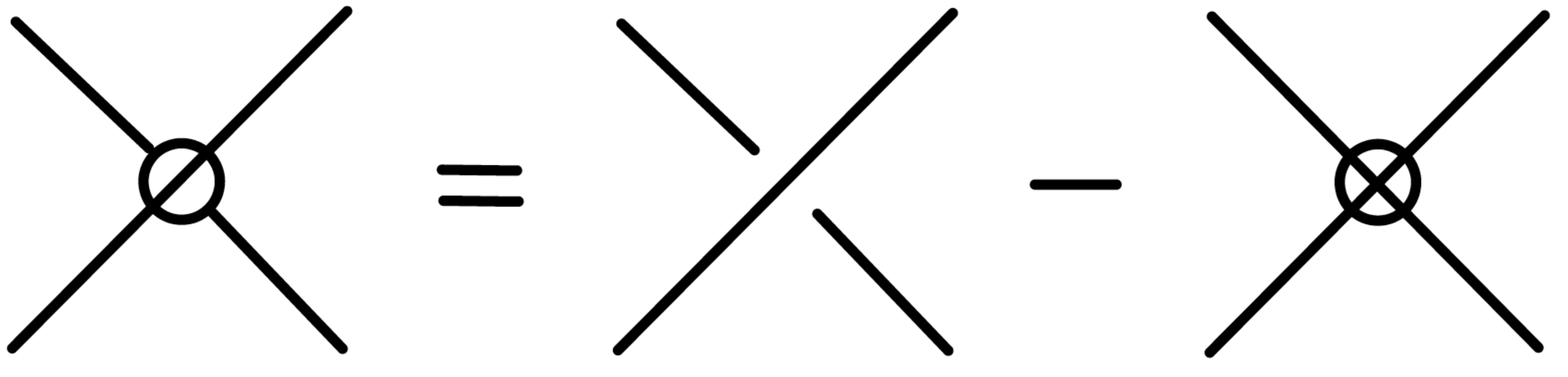}.}
\end{picture}
\end{equation*}
Then Goussarov, Polyak, and Viro formulated Conjecture~\ref{gpv_conj}\footnote{Polyak \cite{P2} wrote that ``An open (an highly non-trivial) conjecture...".  } 
\cite[Conjecture~3.C]{gpv}.     

\begin{conjecture}[{\cite[Conjecture~3.C]{gpv}}]\label{gpv_conj}
Every finite type invariant of classical knots can be extended to a 
finite type invariant of long virtual knots.  
\end{conjecture}
{\color{black}{
First, by \cite{gpv}, it is known that for degrees at most 3, Conjecture~\ref{gpv_conj} is true.  
Here, note that \cite[Section~3.2]{gpv} contains a typo, a Gauss diagram formula of degree $3$ of long virtual knots, it is not a classical knot invariant (see Fig.~\ref{counterEg} and Remark~\ref{rmkGPV}).
}}

\begin{figure}[h!]
\includegraphics[width=8cm]{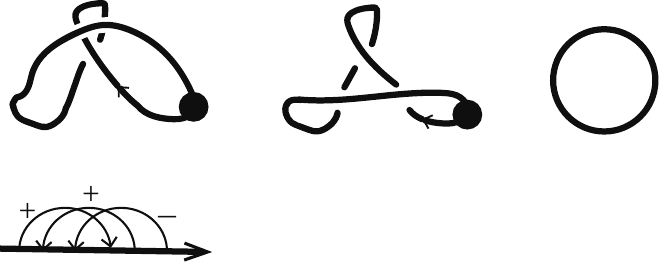}
\caption{A diagram (upper line) of the unknot and its Gauss diagram (lower line) showing that the formula in {\cite[The formula in the end of Section~3.2]{gpv}} is not invariant under Reidemeister moves.  
{\cite[The formula in the end of Section~3.2]{gpv}} takes values $-3$   whereas the value should be $0$ for the unknot.  Each of the terms of $\atraf$, $\dxbpm$, and $\dxbmp$ in their invariant \cite{gpv} takes $-1$.}
\label{counterEg}
\end{figure}     

%
{\color{black}{
The first author discussed Viro about the conjecture and was suggested the following problems related to the conjecture.
\begin{problem}[\cite{viroPrivate}]\label{viroprivate}
$($1$)$ Clarify a relationship between Gauss diagram formulas as long virtual knot invariants and Gauss diagram formulas as classical knot invariants (e.g., Polyak-Viro formula for the degree 3).  \\
$($2$)$ How to merge Gauss diagram formulas as long virtual knot invariants into Gauss diagram formulas as classical knot invariants (e.g., Polyak-Viro formula for the degree 3)? \\
$($3$)$ How to transform Gauss diagram formulas as classical knot invariants into Gauss diagram formulas as long virtual knot invariants?
\end{problem}
We give one solution of them (Theorem~\ref{thm3} and Corollary~\ref{PVformula}) {\color{black}{of degree three}} by giving new Gauss diagram formulas for classical long knots (Proposition~\ref{thm2}).

\begin{theorem}\label{thm3}
Let the orders of $168$ arrow diagrams be as in Notation~\ref{order}.  Let $v_i$ $(1 \le i \le 23)$ $(\tilde{v}_j$ $(1 \le j \le 9)$,~resp.$)$ be as in Proposition~\ref{thm2} $($Proposition~~\ref{thm1},~resp.$)$.  
If $\bf{v}$ is the $23 \times 168$ matrix ${}^t \, (v_{3, 1}, v_{3, 2}, \ldots, v_{3, 21}, v_{2, 1}, v_{2, 2})$ and $\bf{w}$ is the $9 \times 168$ matrix ${}^t \, (\tilde
{v}_{3, 1}, \tilde
{v}_{3, 2}, \ldots, \tilde
{v}_{3, 7}, \tilde
{v}_{2, 1}, \tilde
{v}_{2, 2})$ consisting of the coefficients,    
there exists a unique $9 \times 23$ matrix $A$ such that 
$ A \cdot {\bf{v}} = {\bf{w}}. $
\end{theorem}
 }}

%

Moreover, as a corollary of Proposition~\ref{thm2}, we give a process that converts a Gauss diagram formula of long virtual knots into   
a well-known Polyak-Viro formula \cite[Theorem~2]{PV}   $\left\langle  2 \atrb + \aha, \ \cdot \right\rangle$, which takes value $0$ on the unknot, $+2$ on the right trefoil and $-2$ on the left trefoil.  This Polyak-Viro formula is of degree three and a knot invariant of classical knots and is not invariant of (long) virtual knots (Remark~\ref{rmk_triv}).    

{\color{black}{
\begin{corollary}\label{PVformula}
For $v_{3, 7}$ in the list of Proposition~\ref{thm2}, the formula $v_{3, 7} (\cdot)$ $=$ $\left\langle 2 \atrb + \aha, \ \cdot \right\rangle$ holds.    
\end{corollary} 
}}

The proof of Theorem~\ref{thm3} {\color{black}{ is given in  Section~\ref{PrTh3} and Corollary~\ref{PVformula} is directly implied by Proposition~\ref{thm2}. }} 
Note that Theorem~\ref{thm3} {\color{black}{ and Corollary~\ref{PVformula}}} with Propositions~\ref{thm1} and \ref{thm2} imply that not only the degree three finite type \emph{long virtual} knot   invariants give rise to the same degree finite type \emph{classical}  knot invariant by just taking the restrictions  (as in Proposition~\ref{thm1} of Section~\ref{proofThm1}), but also the Polyak-Viro formula $\Biggl< 2 \atrb +  \aha, \ \cdot \ \Biggl>$ is equals  $v_{3, 7}$ and {\color{black}{ every degree three finite type long virtual knot invariant $\tilde{v}_{3, i}$ ($1 \le i \le 7$) is  combinatorially represented as linear sums of $v_{3, i}$ ($1 \le i \le 21$) and $v_{2, i}$ ($1 \le i \le 2$).  }}

In the end of this section, using Table~\ref{tableSum}, we compare known results with our result (Remarks~\ref{CompareB}--\ref{rmk_triv}).  Here, we identify a (virtual) knot having a base point with a long (virtual) knot.    
\begin{table}
\caption{Known explicit Gauss diagram formulas, dimensions of finite type invariants,  and our Gauss diagram formulas (Propositions~\ref{thm1} and \ref{thm2}).  Proposition~\ref{thm2} is obtained from a quotient of a (modified) Polyak algebra (Section~\ref{redSec})}\label{tableSum}
\renewcommand{\arraystretch}{1.5}
\noindent
\begin{tabular}{|c|c|l|}\hline
 & deg 2 & \multicolumn{1}{c|}{deg 3}   \\ \hline
classical knot&  $\langle \uax, \cdot \rangle$  & Polyak-Viro formula~\ $\langle {2 \atrb + \aha}, \cdot \rangle$ \cite{PV},  which  \\
&$=\langle \bax, \cdot \rangle$& 
is not invariant of virtual knots (Remark~\ref{rmk_triv}), 
 \\
&\cite{PV}& Willerton formula for classical knots \cite{Willerton}, \\  
&& Chmutov-Polyak formula  for classical knots  \cite{CP} \\ 
&& (i.e. \emph{three Gauss diagram formulas were  known}). \\
$\dim$ & $1$ & \multicolumn{1}{c|}{$1$} \\ \cline{2-3}
 &$v_{2, 1}=v_{2, 2}$& 16 (3,~resp.) Gauss diagram formulas, which are \\
Formulas in&$=\langle \uax, \cdot \rangle$& nontrivial (trivial,~resp.) for classical knots  \\ 
Proposition~\ref{thm2}&$=\langle \bax, \cdot \rangle$& and induces $v_{3, 9}=$ $\langle {\aha  +  2 \atrb}, \cdot \rangle$.  \\ \hline
virtual knot& none \cite{gpv} & Goussarov-Polyak-Viro formula \cite{gpv}, which is \\
&&  trivial for classical knots  \\  
&& (i.e. \emph{only 1 Gauss diagram formula was known}). \\
$\dim$ & $0$ \cite{gpv, Bar} & \multicolumn{1}{c|}{$1$ \cite{Bar}}  \\ \hline
long virtual knot & $\langle \uax, \cdot \rangle$, &  Goussarov-Polyak-Viro formula \cite{gpv, CDM, EHLN} (Remark~\ref{rmkGPV}) \\  
&$\langle \bax, \cdot \rangle$ \cite{gpv} & (i.e. \emph{only 1 Gauss diagram formula was known}). \\
$\dim$ & {\color{black}{$2$}} \cite{Bar} & \multicolumn{1}{c|}{$7$ \cite{Bar}}  \\ \cline{2-3}
Seven formulas  && 4 (3,~resp.) Gauss diagram formula(s), which   \\
in Proposition~\ref{thm1} &&are nontrivial (trivial,~resp.) for classical knots \\ \hline
\end{tabular}
\end{table}
   
\begin{remark}\label{CompareB}
By Proposition~\ref{thm1}, we essentially obtain six new invariants of long virtual knots since a known Goussarov-Polyak-Viro finite type invariant of degree three is only one \cite{EHLN}, which is a revised version of \cite{gpv} having typos (see Remark~\ref{rmkGPV}); this known invariant is represented by a linear combination of $\tilde{v}_{3,i}$ $(1 \le i \le 7)$.  
For degree $n$ $=$ $1, 2, \dots, 6$, D.~Bar-Natan, I.~Halacheva, L.~Leung, and F.~Roukema \cite{Bar} computed the dimension of the space of Goussarov-Polyak-Viro finite type invariants.  
In particular, Proposition~\ref{thm1} corresponds to their result that the dimension of the space of Goussarov-Polyak-Viro finite type invariants of degree three is seven  \cite{Bar}.
\end{remark}
\begin{remark}\label{rmkGPV}
Goussarov, Polyak, and Viro gave an example of Goussarov-Polyak-Viro finite type invariants of degree three for long virtual knots by a Gauss diagram formula, and they mentioned that the restriction of this invariant to classical knots was a Vassiliev knot  invariant of the same degree \cite[Section~3.2, Page
~1059]{gpv}.
However, this formula contains a typo.  
In fact, it is not invariant 
under a third Reidemeister move for the unknot {\color{black}{(Fig.~\ref{counterEg})}}.  A correction is given by Chmutov-Duzhin-Mostovoy \cite{CDM} for the typo.  This correction is different from another correction  of Zohar-Hass-Linial-Nowik \cite{EHLN}.   Although, unfortunately, both corrections of \cite{CDM, EHLN} are given without proofs, we have checked 
both formulas by our computation.  See Section~\ref{Tip}.     
Further, for classical knots, we should mention that there are three other works on Vassiliev knot invariants of degree three by Willerton \cite{Willerton}, Fiedler-Stoimenow \cite{FS}, and Chmutov-Polyak \cite{CP}.     
\end{remark}
\begin{remark}\label{rmk_triv}
The Polyak-Viro formula $\Biggl<  2 \atrb + \aha, \ \cdot \ \Biggl>$ is not invariant of virtual knots.  For example, it is elementary to choose a diagram having exactly $3$ real crossings and exactly one virtual crossing of a virtual trefoil (Fig.~\ref{VT}, left).  This diagram takes value $-1$ on $\langle \aha, \cdot \rangle$ and $0$ on $\langle \atrb, \cdot \rangle$.  On the other hand, the virtual knot diagram of the virtual trefoil is changed into another diagram having exactly two real crossings and one virtual crossing (Fig.~\ref{VT}, right).  Then,  
 it takes value $0$ on $\Biggl< 2 \atrb + \aha, \ \cdot \ \Biggl>$. 
It implies that it is not invariant for the virtual trefoil.  
\end{remark}
\begin{figure}[h!]
\includegraphics[width=8cm]{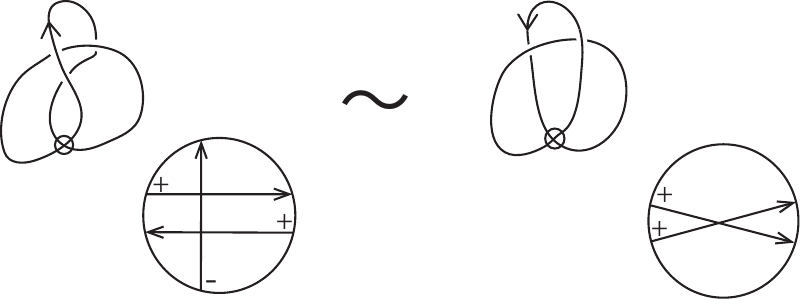}
\caption{A virtual trefoil diagram (left) taking $-1$ on Polyak-Viro formula and the other (right) taking $0$ with their arrow diagrams}\label{VT}
\end{figure}

\emph{Plan of the paper.} 
For readers, we explain the plan of this paper.  In Section~\ref{sec2}, we give definitions of arrow diagrams and their equivalence.  
If the reader is familiar with Gauss diagrams, a space ($\mathbb{Z}$-module or $\mathbb{Q}$-vector space) of them, and its dual space, in  
Section~\ref{sec2}, s/he checks notations only and skip over Definitions~\ref{dfn2_arrow}--\ref{tilde_ll}.   
Definitions~\ref{def_relators_arrow}--\ref{remarkGauss} are needed for a general treatment of Gauss diagram formulas.  Roughly speaking, we  select not only a traditional tuple of relators in \cite{gpv} but also another choice of  relators to produce Gauss diagram formulas that are (long) virtual/classical knot invariants.      
In Section~\ref{proofInvariant_virtual}, we explain our process to obtain Gauss diagram formulas by giving a proof (if the reader is familiar with \cite{Ito_sp}, s/he is at an advantage; however, this paper is a self-contained treatment of the construction of invariants).  
By Section~\ref{redSec}, we introduce a (really) new framework to give Gauss diagram formulas of classical knots using the linking number that is  a Vassiliev-type link invariant.   
In Section~\ref{SecNotation}, we prepare notations to fix orders of Gauss diagrams and relators.  In Section~\ref{proofThm1}, we provide Gauss diagram formulas of long virtual knots, which implies that  a positive answer to the degree three case of Conjecture~\ref{gpv_conj} and in  
Section~\ref{proofThm2}, we give Gauss diagram formulas of classical knots.   
In Section~\ref{PrTh3}, we prove the main result (Theorem~\ref{thm3}) that is a special answer to the degree three case of  Problem~\ref{viroprivate} to solve   Conjecture~\ref{gpv_conj} in a deeper sense.    

\section{Preliminaries}\label{sec2}
In this section, definitions and  notations are based on \cite{Ito}.    
\begin{definition}[arrow diagram \cite{gpv}]\label{dfn2_arrow}  
An \emph{arrow diagram} is a configuration of $n$ pair(s) of points up to ambient isotopy on an oriented circle where each pair of points consists of a starting point and an end point and where the circle has the standard (counterclockwise) orientation and  each pair has a sign.         
If an arrow diagram has a base point which is on the circle and which does not coincide with one of the paired points, the arrow diagram is also called a \emph{based arrow diagram}.  If there is no confusion, a based arrow diagram is simply called an \emph{arrow diagram}.   
If an arrow diagram is a configuration of $n$ pair(s) of points, the integer $n$ is called the {\it{length}} of the arrow diagram.   
Traditionally, two points of each pair are connected by a (straight) arc, and an assignment of starting and end points on the boundary points of an arc is represented by an arrow on the arc from the starting point to the end point.  The arc is called an \emph{arrow}. 
\end{definition}
\begin{remark}[a terminological remark]\label{remark:terminology}
In this paper, we use a simple notation called an \emph{unsigned arrow diagram} $x$ that is the sum of over all ways to assign signs to the arrows, 
 i.e.,  
\[
x := \sum_{x^* : x~{\text{with~signs}} } x^*.   
\]
For example, 
\[
 \dxb =  \dxbp + \dxbmp + \dxbpm + \dxbm.
\] 
The definition is the same as that of \cite{gpv} essentially.  
\end{remark}
\begin{definition}[Reidemeister moves on arrow diagrams \cite{gpv, ostlund}]
\begin{figure}[h]
\includegraphics[width=4cm]{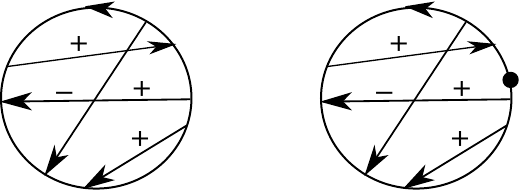}
    \caption{An arrow diagram and a based arrow diagram}
    \label{gaussdiagram}
\end{figure}

{Reidemeister moves} on arrow diagrams are the following three moves as in Fig.~\ref{Reidemeistermoves}. 
If we treat based arrow diagrams, we suppose that there exists a base point on a dashed arc of the circle in Fig.~\ref{Reidemeistermoves}.  \begin{figure}[h]
  \begin{center}
\includegraphics[width=.6\linewidth]{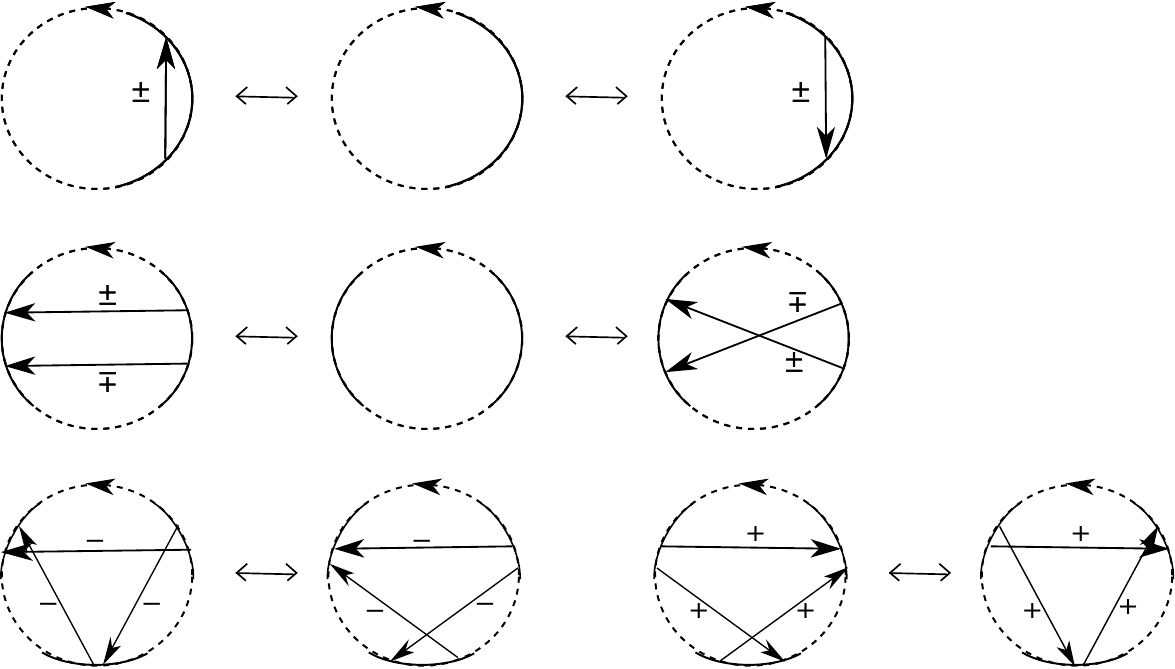}
    \caption{The Reidemeister moves}
    \label{Reidemeistermoves}
  \end{center}
\end{figure}

\end{definition}
\begin{definition}[an arrow diagram of a knot \cite{gpv}]
Let $K$ be a knot and let $D_K$ be a knot diagram of $K$.  Then, there exists a generic immersion $g: S^1 \to \mathbb{R}^2$ such that $g(S^1)=D_K$ which is enhanced by information on the overpass and the underpass at each double point.  We define a based arrow diagram of $D_K$ as follows (e.g. Fig.~\ref{02}).
\begin{figure}[h!]
\includegraphics[width=8cm]{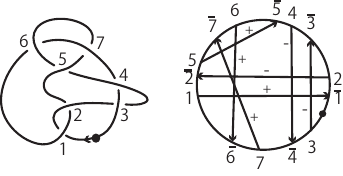}
\caption{A knot diagram (left) and the based arrow diagram (right)}\label{02}
\end{figure}
Let $k$ be the number of the crossings of $D_K$.  Fix a base point, which is not a crossing, and choose an orientation of $D_K$.  We start from the base point and proceed along $D_K$ according to the orientation of $D_K$.  We assign $1$ to the first crossing that we encounter.  Then we assign $2$ to the next crossing that we encounter provided it is not the first crossing.    We suppose that we have already assigned $1$, $2$, \ldots, $p$.   Then we assign $p+1$ to the next crossing that we encounter provided it has not been assigned yet.  Following the same procedure, we finally label all the crossings.  Note that $g^{-1}(i)$ consists of two points on $S^1$ and we shall assign $i$ to one point corresponding to the over path and assign $\bar{i}$ to another point corresponding to the under path.   We shall assign a plus (minus, resp.) to each $g^{-1}(i)$ if the crossing labeled by $i$ is positive (negative, resp.).  
Note also that the preimage of the base point is a point on $S^1$.  
The based arrow diagram represented by the preimage of the base point, $g^{-1}(1)$, $g^{-1}(2)$, \ldots, $g^{-1}(k)$ on $S^1$, equipped with information of the sign of each crossing, is denoted by $AD_{D_K}$, and is called a \emph{based arrow diagram of the knot} $K$.   Here, by ignoring the base point, we have an arrow diagram, which is also called an \emph{arrow diagram of the knot} $K$.   
\end{definition}
It is known that a (\emph{long}) \emph{virtual knot} is identified with the equivalence class of a (based) arrow diagram by the  Reidemeister moves.    
In this paper, we use the terminology of an arrow diagram for either an arrow diagram or a based arrow diagram unless otherwise denoted.  
\begin{definition}[Gauss word \cite{turaev}]\label{d1}
Let $\hat{n}$ $=$ $\{1, 2, 3, \dots, n\}$.  
A {\it{word}} $w$ of length $n$ is a map $\hat{n}$  $\to$ $\mathbb{N}$.  The word is represented by $w(1)w(2)w(3) \cdots w(n)$.  
For a word $w : \hat{n}$ $\to$ $\mathbb{N}$, each element of $w(\hat{n})$ is called a {\it{letter}}.   Each letter has a sign.  The sign of a letter $k$ is denoted by $\sign (k)$.      
A word $w$ of length $2n$ is called a {\it{Gauss word}} of length $2n$ if each letter appears exactly twice in $w(1)w(2)w(3) \cdots w(2n)$.        
Let $\cyc$ and $\rev$ be maps $\hat{2n} \to \hat{2n}$ satisfying that $\cyc(p) \equiv p+1$ (mod $2n$) and $\rev(p) \equiv -p+1$ (mod $2n$).   
\end{definition}
\begin{definition}[oriented Gauss word \cite{Ito}]\label{ori_g}  
Let $v$ be a Gauss word.  For each letter $k$ of $v$, we distinguish the two $k$'s in $v$ by calling one $k$ a {\it{tail}} and the other a {\it{head}}.  
We express the assignments by adding extra informations to $v$ $=$ $v(1)v(2) \cdots v(2n)$, i.e., we add ``$\bar{~}$'' on the letters that are assigned heads.  This new word $v^*$ is called an {\it{oriented Gauss word}}.   We call each letter of an oriented Gauss word an {\it{oriented letter}}.      
Let $v^*$ ($w^*$, resp.) be an oriented Gauss word of length $2n$ induced from $v$ ($w$, resp.).  
Without loss of generality, for $v$, we may suppose that the set of the letters in $v(\hat{2n})$ is $\{ 1, 2, \dots, n \}$.   Here, it is clear that $v^*$ is a word of length $2n$ with letters $\{ 1, 2, \ldots, n, \overline{1}, \overline{2}, \ldots, \overline{n} \}$.     
Two oriented Gauss words $v^*$ and $w^*$ are {\it{isomorphic}} if there exists a bijection $f : v(\hat{2n})$ $\to$ $w(\hat{2n})$ such that $\sign (k)$ $=$ $\sign(f(k))$ and $w^* = f^* \circ v^*$, where $f^* :$ $v^*(\hat{2n})$ $=$ $\{ 1, 2, \ldots, n, \overline{1}, \overline{2}, \ldots, \overline{n} \}$ $\to$ $w^*(\hat{2n})$ is the bijection such that $f^*(i)$ $=$ $f(i)$ and $f^*(\overline{i})$ $=$ $\overline{f(i)}$ ($i=1, 2, \dots, n$).        
Two oriented Gauss words $v^*$ and $w^*$ are {\it{cyclically isomorphic}} if there exists a bijection $f : v(\hat{2n})$ $\to$ $w(\hat{2n})$ such that $\sign (k)$ $=$ $\sign(f(k))$ and
 there exists $t \in \mathbb{Z}$ satisfying that $w^* \circ (\cyc)^{t} = f^* \circ v^*$, where $f^* :$ $v^*(\hat{2n})$ $=$ $\{ 1, 2, \ldots, n, \overline{1}, \overline{2}, \ldots, \overline{n} \}$ $\to$ $w^*(\hat{2n})$ is the bijection such that $f^*(i)$ $=$ $f(i)$ and $f^*(\overline{i})$ $=$ $\overline{f(i)}$ ($i=1, 2, \dots, n$), and where 
$\cyc :$ $\hat{2n} \to \hat{2n}$ is as in Definition~\ref{d1}.    
The isomorphisms and the cyclic isomorphisms obtain equivalence relations on the oriented Gauss words, respectively.     
For an oriented Gauss word $v^*$ of length $2n$, $[[v^*]]$ denotes the equivalence class (the cyclic equivalence class, resp.) of $v^*$.   
In this paper, we use $[[\cdot]]$ as the equivalence class or the cyclic equivalence class depending on the context.  
For $v^*$, an oriented Gauss word ${{v'}^*}$ is called an {\emph{oriented sub-Gauss word}} of the oriented Gauss word $v^*$ if ${{v'}^*}$ is ${v^*}$ itself, or it is obtained from $v^*$ by ignoring some pairs of letters.  The set of the oriented sub-Gauss words of $v^*$ is denoted by $\sub(v^*)$.  For a given oriented Gauss word $v^*$ of length $2n$, let $v^*_{\rev}$ be the oriented Gauss word such that $v^*_{\rev}(i)$ $=$ $v^*(\rev(i))$ ($i \in \hat{2n}$).   It implies a map $v^*$ $\mapsto$ $v^*_{\rev}$ on a set of oriented Gauss words and the map is denoted by $\rev^*$.   
\end{definition} 
In the rest of this paper, we suppose that each letter of an oriented Gauss word has a sign and then every isomorphism class is canonically split into fine isomorphism class.  Under this assumption, however, if there is no confusion, we also call it oriented Gauss word and use the (cyclic)  isomorphisms and (cyclic)  equivalences.  
We note that the equivalence classes (the cyclic equivalence classes, resp.) of the oriented Gauss words of length $2n$ have one to one correspondence with the based arrow diagrams (the  arrow diagrams), each of which has $n$ arrows.  
We identify these four expressions and freely use either one of them depending on situations (Fig.~\ref{four_arrow}). In general, we note that $v^*$ and $v^* \circ \rev$ are not isomorphic.  This implies that an arrow diagram and its reflection image are not ambient isotopic in general.

\begin{figure}[htbp]
\begin{center}
\includegraphics[width=12cm]{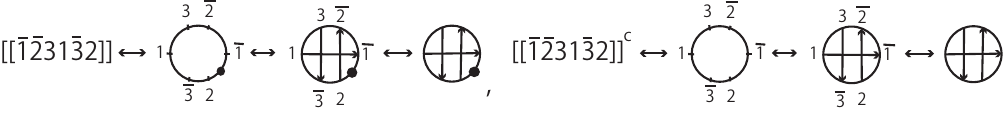}
\caption{Four expressions (signs are omitted)}
\label{four_arrow}
\end{center}
\end{figure} 
\begin{notation}[$\check{G}_{\le d}$, $\check{n}_d$, $\check{G}_{b, d}$]\label{not4}
Let $\check{G}_{< \infty}$ be the set of the arrow diagrams, i.e., the set of isomorphism classes of the oriented Gauss words.  Since $\check{G}_{< \infty}$ consists of countably many elements, there exists a bijection between $\check{G}_{< \infty}$ and $\{ {x}^*_i \}_{i \in {\mathbb{N}}}$, where ${x}^*_i$ is a variable.    
Choose and fix a bijection $\check{f} : \check{G}_{< \infty}$ $\to$ $\{ {x}^*_i \}_{i \in {\mathbb{N}}}$ satisfying: the number of arrows of $\check{f}^{-1}({x}^*_i)$ is less than or equal to that of $\check{f}^{-1}({x}^*_j)$ if and only if $i \le j$ $(i, j \in {\mathbb{N}})$.  For each positive integer $d$, let $\check{G}_{\le d}$ be the set of arrow diagrams each consisting of at most $d$ arrows.  
Let $\check{n}_d$ $=$ $|\check{G}_{\le d}|$.  It is clear that $\check{f}|_{\check{G}_{\le d}}$ is a bijection from $\check{G}_{\le d}$ to $\{ {x}^*_1, {x}^*_2, \ldots, {x}^*_{\check{n}_d} \}$.  
For each pair of integers $b$ and $d$ ($2 \le b \le d$), let $\check{G}_{b, d}$ $=$ $\check{G}_{\le d} \setminus \check{G}_{\le b-1}$.
Then $\check{f}|_{\check{G}_{b, d}}$ is a bijection $\check{G}_{b, d}$ $\to$ $\{ {x}^*_{\check{n}_{b-1} +1}, {x}^*_{\check{n}_{b-1} + 2}, \dots, {x}^*_{\check{n}_d} \}$.  
\end{notation}
 
In the rest of this paper, we use the notations in Notation~\ref{not4} unless otherwise denoted, and we freely use this identification between $\check{G}_{< \infty}$ and $\{{x}^*_i \}_{i \in \mathbb{N}}$.  

Recall that $AD_{D_K}$ gives an equivalence class of oriented Gauss words, say $[[v^*_{D_K}]]$.  Then, by the definition of the equivalence relation, it is easy to see that the map $D_K \mapsto [[v^*_{D_K}]]$ is well-defined.
\begin{definition}[$x^{*}(AD)$]\label{ad_def}
Let $x^* \in \{ x^*_i \}_{i \in \mathbb{N}}$.  For a given arrow diagram $AD$, fix an oriented Gauss word $G^*$ representing $AD$.  
Let $\sub_{x^*} (G^*)$ $=$ $\{H^*~|~H^* \in \sub(G^*)$, $[[H^*]]=x^* \}$. 
Recall that each letter has a sign $\in \{ \pm 1 \}$.  For an oriented Gauss word $H^* \in \sub_{x^*} (G^*)$, the sign $\sign(H^*)$ defined by 
\[
\sign(H^*) = \prod_{\alpha : {\textrm{letter in}}~H^*}~\sign(\alpha).
\] 
Then, 
$\sum_{H^* \in \sub_{x^*} (G^*)} \sign(H^*)$
is denoted by $x^* (G^*)$.     
Let ${G'}^*$ be another oriented Gauss word representing $AD$.   By the definition of the (cyclically) isomorphism of the Gauss words, it is easy to see $x^*({G'}^*)$ $=$ $x^*(G^*)$.  Hence, we shall denote this number by $x^*(AD)$.  
If $AD$ is an arrow diagram of a knot $K$, then $x^*(AD)$ can be denoted by $x^*(D_K)$.
\end{definition}
We note that $x^*(AD)$ is calculated by geometric observations.  We give some examples below (Example~\ref{example2}).   
\begin{example}\label{example2}
$\dxup \left(\ \ahaas \ \  \right)$ $=$ $1$, $\dxump \left(\ \ahaam \ \  \right)$ $=$ $0$, and $\dxupm \left(\ \ahaam \ \  \right)$ $=$ $-1$.
\end{example}
\begin{definition}[$\tilde{x^*} (z)$, $\tilde{x^*} ({[[z]]})$]\label{tilde_ll}
For an arrow diagram $x^*$,  
we define the function $\tilde{x^*}$ from the set of the oriented Gauss words to $\{ -1, 0, 1 \}$ by
 \[
{\tilde{x^*}} (F^*) = \begin{cases}
\sign(F^*) \qquad ~{\textrm{if}}~[[F^*]] = x^* \\
0 \qquad\qquad ~{\textrm{if}}~[[F^*]] \neq x^*.
\end{cases}
\] 
By definition, it is easy to see $\tilde{x^*}(F^*_1)$ $=$ $\tilde{x^*}(F^*_2)$ for each pair $F^*_1$, $F^*_2$ with $[[F^*_1]]$ $=$ $[[F^*_2]]$.  Thus, we shall denote this number by $\tilde{x^*}([[F^*_1]])$.  If $[[F^*]]$ corresponds to an arrow diagram of a knot diagram $D_K$ of a knot $K$, then $\tilde{x^*}([[F^*]])$ is denoted by $\tilde{x^*}(D_K)$.  
Let $\mathbb{Z}[\check{G}_{< \infty}]$ be the free $\mathbb{Z}$-module generated by the elements of $\check{G}_{\le l}$, where $l$ is sufficiently large (note that we always consider finite sums).  
Then, we linearly extend $\tilde{x^*}$ to the function from $\mathbb{Z}[\check{G}_{< \infty}]$ to $\mathbb{Z}$.  
\begin{example}\label{exampletilde}
$\widetilde{\dxup} \left( \dxup \right)$ $=$ $1$, $\widetilde{\dxbpm} \left(\dxbmp \right)$ $=$ $0$, and $\widetilde{\dxbmp} \left(\dxbmp \right)$ $=$ $-1$.
\end{example}
By definition, for any oriented Gauss word $G^*$ with $[[G^*]] = AD$, 
\begin{equation}\label{tilde_x*}
x^*(AD) = \sum_{z^* \in \sub(G^*)} \tilde{x^*} (z^*).
\end{equation}
\end{definition}

We define the elements of $\mathbb{Z}[\check{G}_{< \infty}]$ called {\it{relators}} of types ($\check{\ii}$), ($\check{\s}$), ($\check{\w}$), ($\check{\sss}$), and ($\check{\www}$).  

Before we start to define relators, we note that in Definition~\ref{def_relators_arrow}, if the oriented Gauss word $G^*$ is obtained from a knot diagram $D_K$, then each relator corresponds to a Reidemeister move: Type ($\check{\ii}$) relator to $\mathcal{RI}$, Type ($\check{\s}$) relator to strong~$\mathcal{RI\!I}$, Type ($\check{\w}$) relator to weak~$\mathcal{RI\!I}$, Type ($\check{\sss}$) relator to strong~$\mathcal{RI\!I\!I}$, and Type ($\check{\www}$) relator to weak~$\mathcal{RI\!I\!I}$.  
\begin{figure}[h!]
\includegraphics[width=12cm]{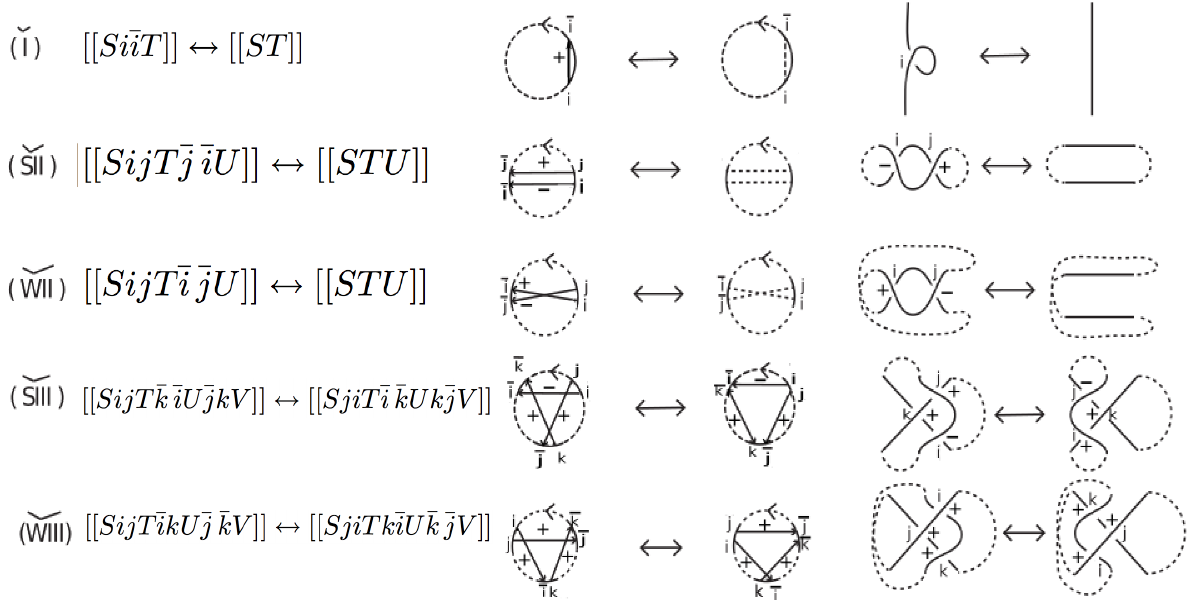}
\caption{Reidemeister moves (the mirror images of the arrow diagrams and the corresponding words are omitted)}\label{def3c}
\end{figure}

Next, we define \emph{relators} for an arrow diagram.  The definition of the case of based arrow diagram is straightforward.  
\begin{definition}[relators (Fig.~\ref{def3c})]\label{def_relators_arrow}
\begin{itemize}
\item Type ($\check{\ii}$). 
An element $r^*$ of $\mathbb{Z}[\check{G}_{< \infty}]$ is called a \emph{Type $(\check{\ii})$ relator} if there exist an oriented Gauss word $ST$ and a letter $i$ not in $ST$ such that $r^*$ $=$ $[[Si\bar{i}T]]$ or $[[S\bar{i}iT]]$ where $\sign(i) \in \{\pm 1\}$.  
\item Type ($\check{\s}$).  
An element $r^*$ of $\mathbb{Z}[\check{G}_{< \infty}]$ is called a \emph{Type $(\check{\s})$ relator} if there exist an oriented Gauss word $STU$ and letters $i$ and $j$ not in $STU$ such that $\sign(i)$ $\neq$ $\sign(j)$ and  $r^*$ $=$ $[[SijT\bar{j}\,\bar{i}U]]$ $+$ $[[SiT\bar{i}U]]$ $+$ $[[SjT\bar{j}U]]$.    
\item Type ($\check{\w}$).  
An element $r^*$ of $\mathbb{Z}[\check{G}_{< \infty}]$ is called a \emph{Type $(\check{\w})$ relator} if there exist an oriented Gauss word $STU$ and letters $i$ and $j$ not in $STU$ such that $\sign(i)$ $\neq$ $\sign(j)$ and 
$r^*$ $=$ $[[SijT\bar{i}\,\bar{j}U]]$ $+$ $[[SiT\bar{i}U]]$ $+$ $[[SjT\bar{j}U]]$.    
\item Type ($\check{\sss}$).  
An element $r^*$ of $\mathbb{Z}[\check{G}_{< \infty}]$ is called a \emph{Type $(\check{\sss})$ relator} if there exist an oriented Gauss word $STUV$ and letters $i$, $j$ and $k$ not in $STUV$ such that  
\begin{align*}
r^* =& \left( [[SijT\bar{k}\,\bar{i}U\bar{j}kV]] + [[SijT\bar{i}U\bar{j}V]] + [[SiT\bar{k}\,\bar{i}UkV]] + [[SjT\bar{k}U\bar{j}kV]] \right) \\
&-  \left([[SjiT\bar{i}\,\bar{k}Uk\bar{j}V]] + [[SjiT\bar{i}U\bar{j}V]] + [[SiT\bar{i}U\,\bar{k}UkV]] + [[SjT\bar{k}Uk\bar{j}V]] \right)
\end{align*}
with $(\sign(i), \sign(j), \sign(k))$ $=$ $(-1,1,1)$ 
or
\begin{align*}
& \left( [[SkjTi\bar{k}U\bar{j}\,\bar{i}V]] + [[SjTiU\bar{j}\,\bar{i}V]] + [[SkTi\bar{k}U\bar{i}V]] + [[SkjT\bar{k}U\bar{j}V]] \right)\\
&-  \left([[SjkT\bar{k}iU\bar{i}\,\bar{j}V]] + [[SjTiU\bar{i}\,\bar{j}V]] + [[SkT\bar{k}iU\bar{i}V]] + [[SjkT\bar{k}U\bar{j}V]] \right) 
\end{align*}
with $(\sign(i), \sign(j), \sign(k))$ $=$ $(1,-1, 1)$.  
\item Type ($\check{\www}$).   
An element $r^*$ of $\mathbb{Z}[\check{G}_{< \infty}]$ is called a \emph{Type $(\check{\www})$ relator} if there exist an oriented Gauss word $STUV$ and letters $i$, $j$, and $k$ not in $STUV$ such that  
\begin{align*}
r^* =& \left( [[SijT\bar{i}kU\bar{j}\,\bar{k}V]] + [[SijT\bar{i}U\bar{j}V]] + [[SiT\bar{i}UkU\bar{k}V]] + [[SjTkU\bar{j}\,\bar{k}V]] \right) \\
&-  \left( [[SjiTk\bar{i}U\bar{k}\,\bar{j}V]] + [[SjiT\bar{i}U\bar{j}V]] + [[SiTk\bar{i}U\bar{k}V]] + [[SjTkU\bar{k}\,\bar{j}V]]  \right).
\end{align*}
with $(\sign(i), \sign(j), \sign(k))$ $=$ $(1,1,1)$
or
\begin{align*}
&\left( [[S\bar{k}\,\bar{j}Tk\bar{i}UjiV]] + [[S\bar{j}T\bar{i}UjiV]] + [[S\bar{k}Tk\bar{i}UiV]] + [[S\bar{k}\,\bar{j}TkUjV]] \right) \\
&-  \left( [[S\bar{j}\,\bar{k}T\bar{i}kUijV]] + [[S\bar{j}T\bar{i}UijV]] + [[S\bar{k}T\bar{i}kUiV]] + [[S\bar{j}\,\bar{k}TkUjV]]  \right),
\end{align*} 
with $(\sign(i), \sign(j), \sign(k))$ $=$ $(1,1,1)$.
\end{itemize}
\end{definition}

We introduce notations in the following.      
Let $D_K$ and $D'_K$ be two knot diagrams of a knot $K$.  
If $D_K$ is related to $D'_K$ by a single $\mathcal{RI}$ (strong~$\mathcal{RI\!I}$, weak~$\mathcal{RI\!I}$, strong~$\mathcal{RI\!I\!I}$, or weak~$\mathcal{RI\!I\!I}$, resp.), then there are Gauss words $G^*$ and ${G'}^*$ such that (by exchanging $D_K$ and $D'_K$, if necessary), $G^*$ $=$ $Si\bar{i}T$ ($S\bar{i}iT$, $SijT\bar{j}\,\bar{i}U$, $SijT\bar{i}\,\bar{j}U$, $SijT\bar{k}\,\bar{i}U\bar{j}kV$, $SkjTi\bar{k}U\bar{j}\,\bar{i}V$, $SijT\bar{i}kU\bar{j}\,\bar{k}V$, or $S\bar{k}\,\bar{j}Tk\bar{i}UjiV$, resp.) and ${G'}^*$ $=$ $ST$ ($ST$, $STU$, $STU$, $SjiT\bar{i}\,\bar{k}Uk\bar{j}V$, $SjkT\bar{k}iU\bar{i}\,\bar{j}V$, $SjiTk\bar{i}U\bar{k}\,\bar{j}V$, or $S\bar{j}\,\bar{k}T\bar{i}kUijV$, resp.) such that $[[G^*]]=AD_{D_K}$ and $[[{G'}^*]]=AD_{D'_K}$ (Fig.~\ref{def3c}).  The subset of $\sub(G^*)$ such that each element has exactly $m$ pairs of oriented letters, each of which arises from $i$, $j$, and $k$ in $G^*$ is denoted by $\sub^{(m)} (G^*)$, where $m=0, 1, 2,$ or $3$.    By definition, we have 
\begin{equation}\label{sub*}
\sub(G^*) = \sub^{(0)}(G^*) \amalg \sub^{(1)} (G^*) \amalg \sub^{(2)} (G^*) \amalg \sub^{(3)} (G^*).   
\end{equation}
Similarly, for an arrow diagram $x^*$, $\sub^{(m)}_{x^*} (G^*)$ denotes the subset of $\sub_{x^*}(G^*)$ consisting of elements, each of which has exactly $m$ pairs of oriented letters, each of which arises from $i$, $j$, and $k$.  Then,  
\begin{equation}\label{subx*}
\sub_{x^*} (G^*) = \sub^{(0)}_{x^*} (G^*) \amalg \sub^{(1)}_{x^*} (G^*) \amalg \sub^{(2)}_{x^*} (G^*) \amalg \sub^{(3)}_{x^*} (G^*).  
\end{equation}

Let $\mathcal{D}$ be the set of (long) virtual knot diagrams.      
Next, for each element $\sum_i \alpha_i x^*_i \in \mathbb{Z}[\check{G}_{< \infty}]$, we define a function $\mathcal{D}$ $\to$ $\mathbb{Z}$, also denoted by $\sum_i \alpha_i x^*_i$.  
\begin{definition}[$\sum_i \alpha_i x^*_i$, $\sum_i \alpha_i \tilde{x}^*_i$]\label{def_xiv}
Let $b$ and $d$ ($2 \le b \le d$) be integers and $\check{G}_{\le d}$, $\check{G}_{b, d}$, and $\{ {x}^*_i \}_{i \in \mathbb{N}}$ as in Notation~\ref{not4}.   Let $\mathbb{Z}[\check{G}_{< \infty}]$ be as in Definition~\ref{tilde_ll}.     
Recall that $\check{n}_d$ $=$ $|\check{G}_{\le d}|$, $\check{G}_{b, d}$ $=$ $\{ {x}^*_i \}_{\check{n}_{b-1} +1 \le i \le \check{n}_d}$, and each $x^*_i$ represents the arrow diagram $\check{f}^{-1}(x^*_i)$ as in Notation~\ref{not4}.  For each 
\[ \sum_{ \cineq } \alpha_i x^*_i \in \mathbb{Z}[\check{G}_{< \infty}], \]
we define an integer-valued function $\mathcal{D}$ $\to$ $\mathbb{Z}$, also denoted by $\cFsum$, by
\[D_K \mapsto \cFsum (D_K),\]
where $x^*_i (D_K)$ is the integer introduced in Notation~\ref{ad_def}.   

Analogously, for each $\cNsum \cat \in \mathbb{Z}[\check{G}_{< \infty}]$, we define the function
\[
\cNsum \cat : \mathbb{Z}[\check{G}_{< \infty}] \to \mathbb{Z}
\]
by
\[
\left( \cNsum \cat \right)(
AD) = \cNsum \cat(AD),  
\]
where $\tilde{x}^*_i (AD)$ is the integer introduced in Definition~\ref{tilde_ll}.
\end{definition}

Then, by using this setting, we can describe the relation between a Type~($\check{\ii}$) relator and a single $\mathcal{RI}$ as follows.  
Suppose that a knot diagram $D_K$ is related to another knot diagram $D'_K$ by a single $\mathcal{RI}$.  Then, we recall that there exist an oriented letter $i$ and an oriented Gauss word $S$ such that $AD_{D_K}$ $=$ $[[Si\bar{i}T]]$ or $[[S\bar{i}iT]]$ and $AD_{D'_K}$ $=$ $[[ST]]$.  Since the arguments are essentially the same, without loss of generality, we may suppose that $T=\emptyset$ and $AD_{D_K}$ $=$ $[[Si\bar{i}T]]$, i.e.,  $AD_{D_K}$ $=$ $[[Si\bar{i}]]$ (every case essentially returns to the argument of the case even if $K$ is a long virtual knot).  
Here, we note that in this case in the decomposition (\ref{sub*}) in the note preceding of Definition~\ref{def_xiv}, $\sub^{(2)}(G^*)$ $=$ $\emptyset$ and $\sub^{(3)}(G^*)$ $=$ $\emptyset$, where $G^*=Si\bar{i}$.  
Then, by (\ref{tilde_x*}) in Definition~\ref{tilde_ll}, and by (\ref{sub*}) and (\ref{subx*}) in the note  preceding of Definition~\ref{def_xiv}, 
\begin{equation*}
\begin{split}
\cFsum &(D_K) = \scFsum \left( \sum_{z^* \in \sub(G^*)} \tilde{x}^*_i (z^*) \right) \\
\!\!\!&= \scFsum \left( \sum_{z^*_0 \in \sub^{(0)}(G^*)} \tilde{x}^*_i (z^*_0) + \sum_{z^*_1 \in \sub^{(1)}(G^*)} \tilde{x}^*_i (z^*_1) \right).    
\end{split}
\end{equation*}
Note that each element $z^*_0 \in$ $\sub^{(0)}(G^*)$ is an oriented sub-Gauss word of $S$.  Then it is clear that $\sub^{(1)}(G^*)$ $=$ $\{ z^*_0 i\bar{i}~|~z^*_0 \in \sub^{(0)}(G^*) \}$.  
Thus, 
\begin{align*}
\cFsum&(D_K) 
= \scFsum \left( \sum_{z^*_0 \in \sub^{(0)}(G^*)} \tilde{x}^*_i (z^*_0) + \sum_{z^*_0 \in \sub^{(0)}(G^*)} \tilde{x}^*_i (z^*_0 i\bar{i}) \right) \\
&= \scFsum \left( \sum_{z^*_0 \in \sub^{(0)}(G^*)} \tilde{x}^*_i ([[z^*_0]]) + \sum_{z^*_0 \in \sub^{(0)}(G^*)} \tilde{x}^*_i ([[z^*_0 i\bar{i}]]) \right).      
\end{align*}
On the other hand, since $\sub({G'}^*)$ is identified with $\sub^{(0)}(G^*)$,
\begin{align*}
\cFsum(D'_K) &= \scFsum \left(\sum_{{z'}^* \in \sub ({G'}^*)} \tilde{x}^*_i ({z'}^*) \right) (\because (\ref{tilde_x*})) \\
&=\scFsum \left(\sum_{z^*_0 \in \sub^{(0)} (G^*)} \tilde{x}^*_i (z^*_0) \right)\\
&= \scFsum \left( \sum_{z^*_0 \in \sub^{(0)} (G^*)} \tilde{x}^*_i \big( [[z^*_0 ]] \big) \right).  
\end{align*}
Here, the last equality is needed to obtain a condition of relators.  
As a conclusion, the difference of the values is calculated as follows:    
\begin{align*}
\cFsum(D_K) &-\cFsum(D'_K)\\ 
&= \cNsum \sum_{z^*_0 \in \sub^{(0)}(G^*)} \alpha_i \tilde{x}^*_i \left( [[z^*_0 i\bar{i}]] \right).
\end{align*}
We note that it is a linear combination of the values of Type~$\check{(\ii)}$ relators via $\tilde{x}^*_i$. 

For the case Type ($\check{\s}$), ($\check{\w}$), ($\check{\sss}$), or ($\check{\www}$) relators, the arguments are slightly more complicated than that of Type ($\check{\ii}$) relator, and we will explain them in Section~\ref{proofInvariant}.   Thus, we omit them here.       
\begin{definition}[$\check{R}_{\epsilon_1 \epsilon_2 \epsilon_3 \epsilon_4 \epsilon_5}$]\label{dfn_relator2}
For each $(\epsilon_1, \epsilon_2, \epsilon_3, \epsilon_4, \epsilon_5) \in \{0, 1\}^{5}$, let $\check{R}_{\epsilon_1 \epsilon_2 \epsilon_3 \epsilon_4 \epsilon_5}$ $=$ $\cup_{\epsilon_i = 1} \check{R}_i$ ($\subset \mathbb{Z}[\check{G}_{< \infty}]$), where $\check{R}_1$ is the set of the Type ($\check{\ii}$) relators (corresponding to $\mathcal{RI}$), $\check{R}_2$ is the set of the Type ($\check{\s}$) relators (corresponding to strong~$\mathcal{RI\!I}$), $\check{R}_3$ is the set of the Type ($\check{\w}$) relators (corresponding to weak~$\mathcal{RI\!I}$), $\check{R}_4$ is the set of the Type ($\check{\sss}$) relators (corresponding to strong~$\mathcal{RI\!I\!I}$), and $\check{R}_5$ is the set of the Type ($\check{\www}$) relators (corresponding to weak~$\mathcal{RI\!I\!I}$).  
\end{definition}
Here, for integers $b$ and $d$ ($2 \le b \le d$), let $\check{O}_{b, d}$ be the projection $\mathbb{Z}[\check{G}_{< \infty}]$ $\to$ $\mathbb{Z}[\check{G}_{b, d}]$.  Note that $\check{O}_{b, d}$ is a linear map. Then, we have the next proposition.     
\begin{proposition}\label{relator_prop*}
For each pair of integers $b$ and $d$ $(2 \le b \le d)$, let $\cNsum \cat$ be a function as in Definition~\ref{def_xiv}.  
For $(\epsilon_1, \epsilon_2, \epsilon_3, \epsilon_4, \epsilon_5)$ $\in \{ 0, 1 \}^5$, let $\cRep$ be the set as in Definition~\ref{dfn_relator2}.
The following two statements are equivalent: 
\begin{enumerate}
\item $\cNsum \alpha_i \tilde{x}^*_i (r^*) = 0 \quad (\forall r^* \in \cRep)$.\label{s1*}
\item $\cNsum \alpha_i \tilde{x}^*_i (r^*) = 0 \quad (\forall r^* \in \check{O}_{b, d}(\cRep))$.\label{s2*}
\end{enumerate}
\end{proposition}
\begin{proof}
Let $r^* \in \cRep$. For $\cineq$, 
\[
\tilde{x}^*_i (r^*) = \tilde{x}^*_i (\check{O}_{b, d} (r^*)).
\]
Therefore,
\[
\cNsum \cat (r^*) = \cNsum \cat (\check{O}_{b, d}(r^*)).  
\]
\end{proof}

\begin{definition}[Polyak algebra \cite{gpv}]\label{factGPV}
The \emph{Polyak algebra} $\mathcal{P}$ is the quotient module $\mathbb{\mathbb{Z}}[\check{G}_{< \infty}] / \cR_{10101}$. Let $n$ be an integer more than one.  
The \emph{truncated Polyak algebra} $\mathcal{P}_n$ is the quotient module $\mathbb{\mathbb{Z}}[\check{G}_{\le n}] / \check{O}_{2, n}(\cR_{10101})$.
\end{definition}
\begin{fact}[\cite{gpv}]\label{factPolyak} 
Let $\mathcal{K}$ be the set of $($long$)$ virtual knots.     
There exists an isomorphism between $\mathbb{Z}[\mathcal{K}]$ and $\mathcal{P}$.  
\end{fact}
\begin{fact}[Polyak \cite{P1}]\label{factPolyak_a}
Reidemeister moves of virtual knots are generated by two types of $\mathcal{RI}$, a type of weak~$\mathcal{RI\!I}$ with the fixed signs, and a type of strong~$\mathcal{RI\!I\!I}$ with the fixed signs.  
\end{fact}
\begin{notation}\label{notation1}
The Reidemeister moves of Fact~\ref{factPolyak_a} is called a \emph{minimal generating set} of Reidemeister moves.  By using Fact~\ref{factPolyak}, there exists a correspondence between the subset of relators $\cR_{10110}$ and  the \emph{minimal generating set} of Reidemeister moves.  The subset is denoted by $\cR_{10110}^{\min}$.  
\end{notation}

Using Notation~\ref{notation1}, it is elementary that  Fact~\ref{factPolyak} together with Definition~\ref{factGPV} implies Fact~\ref{prop_move}.
\begin{fact}\label{prop_move}
There exists an isomorphism $\mathbb{Z}[\mathcal{K}]$ between $\mathbb{\mathbb{Z}}[\check{G}_{< \infty}] / \cR_{10110}^{\min}$.    
\end{fact}  


{\color{black}{In order to simplify descriptions of Gauss diagram formulas, we prepare Definition~\ref{remarkGauss}. }} 
\begin{definition}\label{remarkGauss}
If $\displaystyle \cFsum (\cdot)$ is a knot invariant, it is called a \emph{Gauss diagram formula}. Traditionally, it is represented as a bilinear form $\left\langle \displaystyle \cFsum,~\cdot~\right\rangle$.     
Recall that an unsigned arrow diagram $x$ represents the sum of over all ways to assign signs to the arrows (Remark~\ref{remark:terminology}), 
 i.e.,  
\[
x := \sum_{x^* : x~{\text{with~signs}} } x^*.   
\]
For example, 
\[
\left\langle \dxb, \cdot \right\rangle = \left\langle \dxbp + \dxbmp + \dxbpm + \dxbm, \cdot \right\rangle.
\]
Let $b$, $c$, and $d$ be positive integers ($2 \le b \le c \le d$).  
If $\displaystyle \cFsum$ is obtained by only unsigned arrow diagrams for arrow diagrams $\in \check{G}_{c, d}$, we say that the coefficients $(\alpha_{\check{n}_b + 1}, \alpha_{\check{n}_b + 2}, \dots, \alpha_{\check{n}_d})$ satisfy an unsigned rule for $\mathbb{Z}[\check{G}_{c, d}]$.  
\end{definition}
Here we note the difference: $\widetilde{\dxbmp} \left(\dxbmp \right)$ $=$ $-1$ whereas in \cite{gpv}, $\left(\dxbmp,  \dxbmp \right)$ $=$ $1$.  We use this notation since the unsigned arrow diagram plays a crucial role of this paper.  
\begin{example}
Consider the arrow diagram $AD$ in Fig.~\ref{def5}.  
Fig.~\ref{def5} provides examples of values of functions represented by unsigned arrow diagrams.   
\begin{figure}[h!]
\includegraphics[width=6cm]{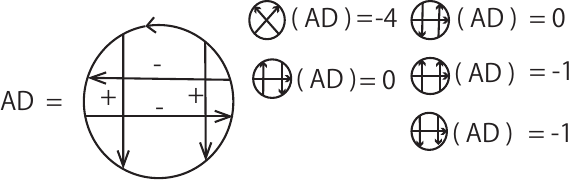}
\caption{Values of functions represented by unsigned arrow diagrams}\label{def5}
\end{figure}
\end{example} 
Proposition~\ref{unsignedProp} means that the unsigned notation is essential to present Vassiliev invariants.  
\begin{proposition}\label{unsignedProp}
Suppose that an arrow diagram $x^*_i$ of the longest length exists  in $\caFsum$.  Then, though there may be a necessary to exchange  order of $\{ x^*_i \}_i$, the longest length arrow diagrams should be  included in terms of the unsigned arrow diagram $x$ $(=$ $\sum_{i=1}^{2^n} x^*_i )$ that is a partial sum of $\caFsum$. 
\end{proposition}
The proof of Proposition~\ref{unsignedProp} will be given by Section~\ref{unsignedProof}.  

\section{Invariances}\label{proofInvariant}
\subsection{Invariances for Gauss diagram formulas  for virtual knots}\label{proofInvariant_virtual}
In this section, in order to prove Proposition~\ref{thm1}, we give Theorem~\ref{gg_thm2}.  The arguments are parallel to those of \cite{Ito_sp}, to which we refer the reader.    A special case of Theorem~\ref{gg_thm2} corresponds to the construction of invariants as in \cite{gpv}.   
\begin{theorem}\label{gg_thm2}
Let $b$ and $d$ $(2 \le b \le d)$ be integers, and let $\check{G}_{\le d}$, $\{ x^*_i \}_{i \in \mathbb{N}}$, $\mathbb{Z}[\check{G}_{< \infty}]$, $\check{n}_d$ $=$ $|\check{G}_{\le d}|$, $\check{G}_{b, d}$ $=$ $\{ x^*_i \}_{\check{n}_{b-1} +1 \le i \le \check{n}_{d}}$, and $\check{O}_{b, d}(\cRep)$ be as in Section~\ref{sec2}.     
Let $\cFsum$ and $\caFsum$ be functions as in Definition~\ref{def_xiv}.  
We arbitrarily choose $(\epsilon_1, \epsilon_2, \epsilon_3, \epsilon_4, \epsilon_5) \in \{0, 1\}^5$ and fix it.  Then, if
$\displaystyle \caFsum(r^*)=0$ for each $r^* \in \check{O}_{b, d}(\cRep)$,  then, $\displaystyle \cFsum$
is an integer-valued invariant under the Reidemeister moves corresponding to $\epsilon_j =1$.  

If $(\epsilon_1, \epsilon_2, \epsilon_3, \epsilon_4, \epsilon_5)$ $=$ $(1, 0, 1, 1, 0)$, $(1, 1, 0, 0, 1)$, or $(1, 1, 0, 1, 0)$, $\displaystyle \cFsum$ is an integer-valued invariant of $($long$)$ virtual knots.    

\end{theorem}
\begin{remark}
Invariants for (long) virtual knots in Theorem~\ref{gg_thm2} imply Goussarov-Polyak-Viro finite type invariants of any degree.      
\end{remark}

Before starting the proof, we note that we use the notion of \emph{sub-words} (Definition~\ref{sub_w}).  
\begin{definition}\label{sub_w}
For a word $w$, a word $u$ of length $q$ is called a \emph{sub-word} of $w$ if there exists an integer $p$ ($q \le p \le n$) such that $u(j)$ $=$ $w(n-p+j)$ ($1 \le j \le q$).
\end{definition}

$\bullet$ (Proof of Theorem~\ref{gg_thm2} for the case $\epsilon_1 =1$.) 
Let $D_K$ and $D'_K$ be two oriented knot diagrams where $D_K$ is related to $D'_K$ by a single $\mathcal{RI}$.  Hence, there exist an oriented letter $i$ and an oriented Gauss word $S$ such that $AD_{D_K}$ $=$ $[[Si\bar{i}T]]$ or $[[S\bar{i}iT]]$ and $AD_{D'_K}$ $=$ $[[ST]]$.  Since the arguments are essentially the same, we may suppose that $T=\emptyset$ and  $AD_{D_K}$ $=$ $[[Si\bar{i}T]]$, i.e, $AD_{D_K}$ $=$ $[[Si\bar{i}]]$ and $AD_{D'_K}$ $=$ $[[S]]$.  
As we observed in Section~\ref{sec2}, we have
\begin{align*}
\cFsum(D_K) &-\cFsum(D'_K) \\
&= \cNsum \sum_{z_0 \in \sub^{(0)}(G^*)} \cat \left( [[z_0 i\bar{i}]] \right).
\end{align*}
By the assumption of this case, for each $z_0 \in \sub^{(0)}(G^*)$, 
\[
\cNsum \cat \left( [[z_0 i\bar{i}]] \right)=0
\]
and this shows that 
\[
\cFsum(D_K) = \cFsum(D'_K).  
\]

$\bullet$ (Proof of Theorem~\ref{gg_thm2} for the case $\epsilon_2 =1$.)  Let $D_K$ and $D'_K$ be two oriented knot diagrams where $D_K$ is related to $D'_K$ by a single strong~$\mathcal{RI\!I}$, hence, there exist two oriented Gauss words $G^*=SijT\bar{j}\, \bar{i}U$ and ${{G'}^*}=STU$ corresponding to $D_K$ and $D'_K$, respectively, i.e., $AD_{D_K}$ $=$ $[[SijT\bar{j}\,\bar{i}U]]$  and $AD_{D'_K}$ $=$ $[[STU]]$.  Since the arguments are essentially the same, we may suppose that $U=\emptyset$ in the following.

By (\ref{tilde_x*}), (\ref{sub*}), and (\ref{subx*}) in Section~\ref{sec2}, we have (note that $\sub^{(3)} (G^*)$ $=$ $\emptyset$): 

\begin{align*}
&\cFsum (D_K) = \scFsum \left(\sum_{z^* \in \sub(G^*)} \tilde{x}^*_i (z^*) \right) \\
&= \scFsum \left(  \sum_{z^*_0 \in \sub^{(0)}(G^*)} \tilde{x}^*_i (z^*_0) \right)
+ \cNsum \sum_{z^*_{12} \in \sub^{(1)}(G^*) \cup \sub^{(2)} (G^*)} \alpha_i \tilde{x}^*_i (z^*_{12})\\
&= \scFsum \left(  \sum_{{z'}^* \in \sub({G'}^*)} \tilde{x}^*_i ({z'}^*) \right)
+ \cNsum \sum_{z^*_{12} \in \sub^{(1)}(G^*) \cup \sub^{(2)} (G^*)} \alpha_i \tilde{x}^*_i (z^*_{12})
\end{align*}
($\because$ $\sub^{(0)}(G^*)$ is identified with $\sub({G'}^*)$).

Let $z^*_0 \in \sub^{(0)}(G^*)$.
Note that since $G^*$ is an oriented Gauss word $z^*_0$ uniquely admits a decomposition into two  sub-words, which are sub-words on $S$ and $T$.
Let $\sigma(z^*_0)$ be the sub-word of $S$ and $\tau(z^*_0)$ the sub-word of $T$ satisfying $z^*_0$ $=$ $\sigma(z^*_0)\tau(z^*_0)$.  By using these notations, we define maps
\begin{align*}
z^*_2: &\sub^{(0)}(G^*) \to \sub^{(2)}(G^*); z^*_2 (z^*_0) = \sigma(z^*_0)ij \tau(z^*_0)\bar{j}\,\bar{i},\\
z^*_1: &\sub^{(0)}(G^*) \to \sub^{(1)}(G^*); z^*_1 (z^*_0) = \sigma(z^*_0)i\tau(z^*_0)\bar{i} ~{\textrm{, and}}\\
{z'}^*_1: &\sub^{(0)}(G^*) \to \sub^{(1)}(G^*); {z'}^*_1 (z^*_0) = \sigma(z^*_0)j\tau(z^*_0)\bar{j}.    
\end{align*}
Then, it is easy to see that $\sub^{(1)} (G^*) \cup \sub^{(2)} (G^*)$ admits a decomposition 
\begin{align*}
&\sub^{(1)} (G^*) \cup \sub^{(2)} (G^*)\\ 
&=  \{ z^*_1 (z^*_0) ~|~ \forall z^*_0 \in \sub^{(0)} (G^*) \} \amalg \{ {z'}^*_1 (z^*_0) ~|~ \forall z^*_0 \in \sub^{(0)} (G^*) \} \\
& \qquad\qquad\qquad\qquad\qquad\qquad\qquad\qquad \amalg \{ z^*_2 (z^*_0) ~|~ \forall z^*_0 \in \sub^{(0)}(G^*) \}.  
\end{align*}
These notations together with the above give: 
\begin{align*}
&\cFsum(D_K) = \cFsum (D'_K) \\
&\qquad\qquad\qquad + \cNsum \sum_{z^*_0 \in \sub^{(0)}(G^*)} \cat (z^*_2 (z^*_0) + z^*_1 (z^*_0) + {z'}^*_1 (z^*_0))\\
&= \cFsum (D'_K) \\
&\qquad + \cNsum \sum_{z^*_0 \in \sub^{(0)}(G^*)} \cat ([[z^*_2 (z^*_0)]] + [[z^*_1 (z^*_0)]] + [[{z'}^*_1 (z^*_0)]]).
\end{align*}
Here, note that by the condition for the case $\epsilon_2$ $=$ $1$, for any for any $z^*_0 \in \sub^{(0)}(G^*)$, 
\[
\cNsum \cat ([[z^*_2 (z^*_0)]] + [[z^*_1 (z^*_0)]] + [[{z'}^*_1 (z^*_0)]]) = 0.
\]
Here, one may think that
\[
\tilde{x}^*_i ([[z^*_1 (z^*_0)]] + [[{z'}^*_1 (z^*_0)]] + [[z^*_2 (z^*_0)]]) = 0
\]
for each $z^*_0$ by the condition of the statement (for the case $\epsilon_2$ $=$ $1$).   
However, the condition means that the equation holds for each $r^* \in \check{O}_{b, d}(\check{R}_{01000})$, and we note that $[[z^*_1 (z^*_0)]]$ $+$ $[[{z'}^*_1 (z^*_0)]]$ $+$ $[[z^*_2 (z^*_0)]]$ may not be an element of $\check{O}_{b, d}(\check{R}_{01000})$ (possibly $\check{f}^{-1}([[z^*_2 (z^*_0)]]) > \check{n}_d$ or $\check{f}^{-1}([[z^*_1 (z^*_0)]]) \le \check{n}_{b-1}$), and $\check{O}_{b, d} ([[z^*_1 (z^*_0)]] + [[{z'}^*_1 (z^*_0)]] + [[z^*_2 (z^*_0)]])$ $\neq 0$.  
However, even when this is the case we see that 
\[
\tilde{x}^*_i ([[z^*_1 (z^*_0)]] + [[{z'}^*_1 (z^*_0)]] + [[z^*_2 (z^*_0)]]) = 0
\]
by Proposition~\ref{relator_prop*}. 

Thus,
\begin{align*}
\cFsum (D_K) = \cFsum (D'_K).
\end{align*}

$\bullet$ (Proof of Theorem~\ref{gg_thm2} for the case $\epsilon_3 =1$.)  
Since the arguments are essentially the same as that of the case $\epsilon_2$ $=$ $1$, we omit this proof. 

$\bullet$ (Proof of Theorem~\ref{gg_thm2} for the case $\epsilon_4 =1$.)  Let $D_K$ and $D'_K$ be two oriented knot diagrams where $D_K$ is related to $D'_K$ by a single strong~$\mathcal{RI\!I\!I}$, hence, there exist two Gauss words $G^*=SijT\bar{k}\,\bar{i}U\bar{j}kV$ ($SkjTi\bar{k}U\bar{j}\,\bar{i}V$, resp.) and ${{G'}^*}=SjiT\bar{i}\,\bar{k}Uk\bar{j}V$ ($SjkT\bar{k}iU\bar{i}\,\bar{j}V$, resp.) corresponding to $D_K$ and $D'_K$, respectively, i.e., $AD_{D_K}$ $=$ $[[SijT\bar{k}\,\bar{i}U\bar{j}kV]]$ ($[[SkjTi\bar{k}U\bar{j}\,\bar{i}V]]$, resp.) and $AD_{D'_K}$ $=$ $[[SjiT\bar{i}\,\bar{k}Uk\bar{j}V]]$ ($[[SjkT\bar{k}iU\bar{i}\,\bar{j}V]]$, resp.). 

First, we suppose that $AD_{D_K}$ $=$ $[[SijT\bar{k}\,\bar{i}U\bar{j}kV]]$.  Since the arguments are essentially the same, we may suppose that $V=\emptyset$ in the following.  
 
By (\ref{tilde_x*}), (\ref{sub*}), and (\ref{subx*}) in Section~\ref{sec2}, we have: 
\begin{align*}
&\cFsum (D_K) = \scFsum \left( \sum_{z^* \in \sub(G^*)} \tilde{x}^*_i (z^*) \right)  \\
&= \scFsum \left( \sum_{z^*_{01} \in \sub^{(0)}(G^*) \cup \sub^{(1)}(G^*)} \tilde{x}^*_i (z^*_{01}) 
+ \sum_{z^*_{23} \in \sub^{(2)}(G^*) \cup \sub^{(3)}(G^*)} \tilde{x}^*_i (z^*_{23})
 \right).  
\end{align*}
and
\begin{align*}
&\cFsum (D'_K) = \scFsum \left( \sum_{{z'}^* \in \sub({G'}^*)} \tilde{x}^*_i ({z'}^*) \right)  \\
&= \scFsum \left( \sum_{{z'}^*_{01} \in \sub^{(0)}(G^*) \cup \sub^{(1)}(G^*)} \tilde{x}^*_i ({z'}^*_{01}) 
+ \sum_{z^*_{23} \in \sub^{(2)}(G^*) \cup \sub^{(3)}(G^*)} \tilde{x}^*_i ({z'}^*_{23})
 \right).
\end{align*}

Because $\sub^{(0)}(G^*)$ ($\sub^{(1)}(G^*)$, resp.) is naturally identified with $\sub^{(0)}({G'}^*)$ ($\sub^{(1)}({G'}^*)$, resp.), the above equations show: 
\begin{align*}
&\cFsum (D_K) - \cFsum (D'_K) \\ &= \cNsum \sum_{z^*_{23} \in \sub^{(2)}(G^*) \cup \sub^{(3)} (G^*)} \cat (z^*) \\ &\qquad\qquad\qquad\qquad\qquad - \cNsum \sum_{{z'}^*_{23} \in \sub^{(2)}({G'}^*) \cup \sub^{(3)} ({G'}^*)} \cat ({z'}^*).   
\end{align*}

Let $z^*_0 \in \sub^{(0)}(G^*)$ that is identified with $\sub^{(0)}({G'}^*)$.  Note that since $G^*$ is an oriented Gauss word $z^*_0$ uniquely admits a decomposition into three sub-words, which are sub-words on $S$, $T$, and $U$.  
Let $\sigma(z^*_0)$ be the sub-word of $S$, $\tau(z^*_0)$ the sub-word of $T$, and $\mu(z^*_0)$ the sub-word of $U$, where $z^*_0$ $=$ $\sigma(z^*_0)\tau(z^*_0)\mu(z^*_0)$.  Let  
\begin{align*}
z^*_3: &\sub^{(0)}(G^*) \to \sub^{(3)}(G^*); z^*_3 (z^*_0) = \sigma(z^*_0) ij \tau(z^*_0) \bar{k}\,\bar{i} \mu(z^*_0) \bar{j}k,\\ 
z^*_{2a}: &\sub^{(0)}(G^*) \to \sub^{(2)}(G^*); z^*_{2a} (z^*_0) = \sigma(z^*_0) ij \tau(z^*_0) \bar{i} \mu(z^*_0) \bar{j}, \\
z^*_{2b}: &\sub^{(0)}(G^*) \to \sub^{(2)}(G^*); z^*_{2b} (z^*_0) = \sigma(z^*_0) i \tau(z^*_0) \bar{k}\,\bar{i} \mu(z^*_0) k ,~{\text{and}}~\\
z^*_{2c}: &\sub^{(0)}(G^*) \to \sub^{(2)}(G^*); z^*_{2c} (z^*_0) = \sigma(z^*_0) j \tau(z^*_0) \bar{k} \mu(z^*_0) \bar{j}k.
\end{align*}
Similarly, we define maps  
\begin{align*}
{z'}^*_3: &\sub^{(0)}({G'}^*) \to \sub^{(3)}({G'}^*); {z'}^*_3 (z^*_0) = \sigma(z^*_0) ji 
 \tau(z^*_0) \bar{i}\,\bar{k} \mu(z^*_0) k\bar{j},\\
{z'}^*_{2a}: &\sub^{(0)}({G'}^*) \to \sub^{(2)}({G'}^*); {z'}^*_{2a} (z^*_0) = \sigma(z^*_0) ji \tau(z^*_0) \bar{i} \mu(z^*_0) \bar{j}, \\
{z'}^*_{2b}: &\sub^{(0)}({G'}^*) \to \sub^{(2)}({G'}^*); {z'}^*_{2b} (z^*_0) = \sigma(z^*_0) i \tau(z^*_0) \bar{i}\,\bar{k} \mu(z^*_0) k,~{\text{and}}~ \\
{z'}^*_{2c}: &\sub^{(0)}({G'}^*) \to \sub^{(2)}({G'}^*); {z'}^*_{2c} (z^*_0) = \sigma(z^*_0) j \tau(z^*_0) \bar{k} \mu(z^*_0) k\bar{j}.
\end{align*}
Here, it is easy to see that $\sub^{(2)} (G^*) \cup \sub^{(3)} (G^*)$ admits decompositions 
\begin{align*}
&\sub^{(2)} (G^*) \cup \sub^{(3)} (G^*) \\
&= \{ z^*_3 (z^*_0) ~|~ z^*_0 \in \sub^{(0)}(G^*) \}
\amalg \{ z^*_{2a} (z^*_0) ~|~ z^*_0 \in \sub^{(0)} (G^*) \} \\
& \amalg \{ z^*_{2b} (z^*_0) ~|~ z^*_0 \in \sub^{(0)} (G^*) \}  
\amalg \{ z^*_{2c} (z^*_0) ~|~ z^*_0 \in \sub^{(0)}(G^*) \}
\end{align*}
and 
\begin{align*}
& \sub^{(2)}({G'}^*) \cup \sub^{(3)} ({G'}^*) \\
&= \{ {z'}^*_3 (z^*_0) ~|~ z^*_0 \in \sub^{(0)}({G}^*) \}
\amalg \{ {z'}^*_{2a} (z^*_0) ~|~ z^*_0 \in \sub^{(0)} ({G}^*) \} \\
& \amalg \{ {z'}^*_{2b} (z^*_0) ~|~ z^*_0 \in \sub^{(0)} ({G}^*) \} 
\amalg \{ {z'}^*_{2c} (z^*_0) ~|~ z^*_0 \in \sub^{(0)} ({G}^*) \}.
\end{align*}
Under these notations, we have: 
\begin{align}\label{rel_eq}
&\cFsum (D_K) - \cFsum (D'_K)
\\ \nonumber
&= \sum_{z^*_0 \in \sub^{(0)}(G^*)}  \caFsum \Big( (z^*_3 (z^*_0) + z^*_{2a} (z^*_0) + z^*_{2b} (z^*_0) + z^*_{2c} (z^*_0))\\ \nonumber
&- ({z'}^*_3(z^*_0) + {z'}^*_{2a}(z^*_0) + {z'}^*_{2b}(z^*_0) + {z'}^*_{2c}(z^*_0)) \Big) \\ \nonumber
&= \sum_{z^*_0 \in \sub^{(0)}(G^*)} \caFsum \Big(  ([[z^*_3 (z^*_0)]] + [[z^*_{2a} (z^*_0)]] + [[z^*_{2b} (z^*_0)]] + [[z^*_{2c} (z^*_0)]])\\ 
& - ([[{z'}^*_3(z^*_0)]] + [[{z'}^*_{2a}(z^*_0)]] + [[{z'}^*_{2b}(z^*_0)]] + [[{z'}^*_{2c}(z^*_0)]]) \Big). \nonumber
\end{align}
Here, note that 
\begin{align*}
&([[z^*_3 (z^*_0)]] + [[z^*_{2a} (z^*_0)]] + [[z^*_{2b} (z^*_0)]] + [[z^*_{2c} (z^*_0)]]) - ([[{z'}^*_3(z^*_0)]] + [[{z'}^*_{2a}(z^*_0)]] \\
&\qquad\qquad\qquad\qquad\qquad\quad\qquad\qquad\qquad\qquad + [[{z'}^*_{2b}(z^*_0)]] + [[{z'}^*_{2c}(z^*_0)]]) \in \check{R}_{00010}.  
\end{align*}
Thus, by the assumption of Case $\epsilon_4$ $=$ $1$ and by Proposition~\ref{relator_prop*} (cf.~Proof of the case $\epsilon_2$ $=$ $1$), for each $z^*_0$, 
\begin{align*}
&\caFsum  \Big(  ([[z^*_3 (z^*_0)]] + [[z^*_{2a} (z^*_0)]] + [[z^*_{2b} (z^*_0)]] + [[z^*_{2c} (z^*_0)]]) \\
&\qquad\qquad\qquad  - ([[{z'}^*_3(z^*_0)]] + [[{z'}^*_{2a}(z^*_0)]] + [[{z'}^*_{2b}(z^*_0)]] + [[{z'}^*_{2c}(z^*_0)]]) \Big)\\
&\qquad = 0.  
\end{align*}
They show that  
\[ \cFsum (D_K) = \cFsum (D'_K).
\]

The proof for the case when $AD_{D_K}$ $=$ $[[SkjTi\bar{k}U\bar{j}\,\bar{i}V]]$ can be carried out as above, and omit it.

$\bullet$ (Proof of Theorem~\ref{gg_thm2} for the case $\epsilon_5 =1$.)  
Since the arguments are essentially the same as that of the case $\epsilon_4$ $=$ $1$, we omit this proof.   

$\hfill\Box$

\begin{definition}[irreducible arrow diagram]
Let $x^*$ be an arrow diagram.  An arrow $\alpha$ in $x^*$ is said to be an \emph{isolated arrow} if $\alpha$ does not intersect any other arrow.  If $x^*$ has an isolated arrow, $x^*$ is called \emph{reducible} and otherwise, $x^*$ is called \emph{irreducible}.           
The set of the irreducible arrow diagrams is denoted by $\check{\irr}$.  Let $\cirri$ $=$ $\{ i ~|~ \cineq, x^*_i \in \check{\irr} \}$. 
\end{definition}

If we consider the function of the form $\cFsumIrr$ for $\cFsum$ in Theorem~\ref{gg_thm2}, we have:  
\begin{corollary}\label{g_thm2}
Let $b$ and $d$ $(2 \le b \le d)$ be integers and let $\check{G}_{\le d}$, $\{ x^*_i \}_{i \in \mathbb{N}}$, $\mathbb{Z}[\check{G}_{< \infty}]$, $\check{n}_d$ $=$ $|\check{G}_{\le d}|$, $\check{G}_{b, d}$ $=$ $\{ x^*_i \}_{\check{n}_{b-1} +1 \le i \le \check{n}_{d}}$ and $\check{O}_{b, d}(\cRep)$ be as in Section~\ref{sec2}.     
Let $\cFsumIrr$ and $\caFsumIrr$ be functions as in Definition~\ref{def_xiv}.       
We arbitrarily choose $(\epsilon_2, \epsilon_3, \epsilon_4, \epsilon_5) \in \{0, 1\}^4$ and fix it.   Then, if $\displaystyle \caFsum(r^*)=0$ for each $r^* \in \check{O}_{b, d}(\check{R}_{0 \epsilon_2 \epsilon_3 \epsilon_4 \epsilon_5})$, $\displaystyle \cFsumIrr$ is an integer-valued invariant under $\mathcal{RI}$ and the Reidemeister moves corresponding to $\epsilon_j =1$.  
If $(\epsilon_2, \epsilon_3, \epsilon_4, \epsilon_5)$ is $(0, 1, 1, 0)$, $(1, 0, 0, 1)$, or $(1, 0, 1, 0)$, $\displaystyle \cFsum$ is an integer-valued invariant of $($long$)$ virtual knots.  

\end{corollary}
\begin{proof}[Proof of Corollary~\ref{g_thm2} from Theorem~\ref{gg_thm2}]
By Theorem~\ref{gg_thm2}, it is enough to show
\begin{equation}\label{eq2_color}
\caFsumIrr (r^*)=0  \quad (\forall r^* \in \check{O}_{b, d}(\check{R}_{10000}))
\end{equation}
for a proof of Corollary~\ref{g_thm2}.  We first note that if $x^*_i \in \check{\irr}$, then $x^*_i$ has no isolated arrows.  On the other hand, let $r^* \in \check{O}_{b, d}(\check{R}_{10000})$, i.e., there exist an oriented Gauss word $S$ and a letter $j$ such that $r^* =[[Sj\bar{j}T]]$ or $[[S\bar{j}jT]]$.  Then, the arrow corresponding to $j$ is isolated.  
They show that $\tilde{x}^*_i (r^*)=0$.  Then, this shows that (\ref{eq2_color}) holds.    
This fact together with Theorem~\ref{gg_thm2} immediately gives Corollary~\ref{g_thm2}.  

\end{proof}



\begin{definition}[connected arrow diagram]\label{defconnected}
Let $v^*$ be an oriented Gauss word of length $2m$ and $w^*$ an oriented Gauss word of length $2n$ satisfying that $v^*(\hat{2m}) \cap w^*(\hat{2n})$ $=$ $\emptyset$.  
Then, for $v^*$ and $w^*$, we define the Gauss word of length $2(m+n)$, denoted by $v^*w^*$, by $v^*w^*(i)=v^*(i)$ ($1 \le i \le 2m$) and $v^*w^*(2m+i)=w^*(i)$ ($1 \le i \le 2n$).  
An arrow diagram is a \emph{connected arrow diagram} if it is not an arrow diagram satisfying that there exist non-empty oriented Gauss words $v^*$ and $w^*$ such that $[[z^*]]$ $=$ $[[v^*w^*]]$.    
Then the set of the connected arrow diagrams is denoted by $\check{\conn}$ and $\{ i ~|~ \cineq, x^*_i \in \check{\conn} \}$ is denoted by $\cconni$.
\end{definition}

\begin{definition}\label{def_conn}
Let $K$ and $K'$ be two knots (long virtual knots, resp.).  A connected sum of $K$ and $K'$ is denoted by $K \sharp K'$.  
Suppose that a function $v$ is an invariant of knots (long virtual knots, resp.).  If $v$ is additive with respect to a connected sum $K \sharp K'$ of $K$ and $K'$, i.e., $v(K \sharp K')$ $=$ $v(K)+v(K')$, then we say that $v$ is \emph
{additive}.  
\end{definition}
If we consider the function of the form $\cFsumConn$ for $\cFsum$ in Theorem~\ref{gg_thm2}, we have:    
\begin{corollary}\label{cor3b}
Let $b$ and $d$ $(2 \le b \le d)$ be integers, and let $\check{G}_{\le d}$, $\{ x^*_i \}_{i \in \mathbb{N}}$, $\mathbb{Z}[\check{G}_{< \infty}]$, $\check{n}_d$ $=$ $|\check{G}_{\le d}|$, $\check{G}_{b, d}$ $=$ $\{ x^*_i \}_{\check{n}_{b-1} +1 \le i \le \check{n}_{d}}$, and $\check{O}_{b, d}(\cRep)$ be as in Section~\ref{sec2}.     
Let $\cFsumConn$ and $\caFsumConn$ be functions as in Definition~\ref{def_xiv}.  We arbitrarily choose $(\epsilon_2, \epsilon_3, \epsilon_4, \epsilon_5) \in \{0, 1\}^5$ and fix it.  Then, if 
$\displaystyle \caFsum(r^*)=0$ for each $r^* \in \check{O}_{b, d}(\check{R}_{0 \epsilon_2 \epsilon_3 \epsilon_4 \epsilon_5})$, 
then, 
$\displaystyle \cFsumConn$
is an integer-valued  additive invariant under $\mathcal{RI}$ and the Reidemeister moves corresponding to $\epsilon_j =1$.  
If $(\epsilon_1, \epsilon_2, \epsilon_3, \epsilon_4, \epsilon_5)$ $=$ $(1, 0, 1, 1, 0)$, $(1, 1, 0, 0, 1)$, or $(1, 1, 0, 1, 0)$, 
 $\displaystyle \cFsum$ is an integer-valued additive  invariant of $($long$)$ virtual knots.   

\end{corollary}
\begin{proof}[Proof of Corollary~\ref{cor3b} from Theorem~\ref{gg_thm2}]
Since $b \ge 2$ and $i \ge \check{n}_{b-1} +1$, each $x^*_i$ consists of more than one arrows, i.e., $x^*_i \in \check{\irr}$ (see the note preceding Definition~\ref{def_conn}). 
Then, the former part of the statement is proved in the same argument as {\it{Proof of Corollary~\ref{g_thm2} from Theorem~\ref{gg_thm2}}}.

Next, by using geometric observations as in Example~\ref{example2}, it is clear that if $x^*_i \in \conn$, then, there exists a non-connected arrow diagram $AD_{D(K \sharp K')}$ consisting of copies of $AD_{D_K}$ and $AD_{D_{K'}}$, 
\begin{equation}\label{eq4a}
{x^*_i}(AD_{D(K \sharp K')}) = {x^*_i}(AD_{D_K}) + {x^*_i}(AD_{D_{K'}}).  
\end{equation}
This fact implies that $\sum_{\substack{ n_{b-1} +1 \le i \le n_d \\ x_i \in \conn } } \alpha_i x_i$ is additive.  

\end{proof}

\subsection{Invariances for Gauss diagram formulas  for classical knots}\label{redSec}
In this section, 
we introduce a framework giving Gauss diagram formulas of  classical knots; we have Proposition~\ref{prop2}.   
{\color{black}{  
Although one may feel the formulation of this truncated  reduced Polyak algebra is complicated, the idea is simple as follows.  
We firstly choose a module generated by finitely many Gauss diagrams.  Secondly, this module is divided by relations corresponding to Reidemeister moves and the linking number relation (\cite[Theorem~5]{PV}, \cite[Section~4.1]{ostlund}): 
\begin{equation}
\langle \vone, \cdot \rangle =  \langle \voneA, \cdot \rangle.  
\end{equation} 
}}
{\color{black}{
Then, intuitively, Definition~\ref{mirror_dfn} defines a relator, written as
\[
\vone - \voneA, 
\]
which corresponds to  
\[
\langle \vone, \cdot \rangle -   \langle \voneA, \cdot \rangle = 0.
\]
\begin{definition}[mirroring pair]\label{mirror_dfn}
Let $S$, $T$, $U$, and $V$ be sub-words.  
Let $r^\epsilon_{ST}$ be a Type~($\check{\sss}$) relator having a single arrow with the sign $\epsilon$ from $T$ to $U$ ($\epsilon$ $=$ $+, -$). 
Let $\hat{r}^\epsilon_{TU}$ $:=$ $r^\epsilon_{UT}$ (e.g.  Fig.~\ref{mir}).  
Further, in general, let $(\cdot, \cdot)$ $=$ $(S, T)$, $(T, S)$, $(S, U)$, $(U, S)$, $(T, U)$ $(U, T)$, $(T, V)$ $(V, T)$, $(U, V)$, or $(V, U)$.
Then the pair $(r^\epsilon_{\cdot \cdot}, \hat{r}^\epsilon_{\cdot \cdot})$ is called the \emph{mirroring pair}.  
Let $(r^\epsilon_{\cdot \cdot}, \hat{r}^\epsilon_{\cdot \cdot})$ be a mirroring pair.  For $\check{O}_{3, 3}(\cR^{\min}_{00010})$, $\reduced(\check{O}_{3, 3}(\cR^{\min}_{00010}))$ is the set of elements, each of which consists of relators $r^\epsilon_{\cdot \cdot} + \hat{r}^\epsilon_{\cdot \cdot}$.    
\end{definition}
}} 
By definition, each mirroring pair consists of $16$ term formulas since every relator has $8$ terms.  
\begin{figure}[h!]
\includegraphics[width=12cm]{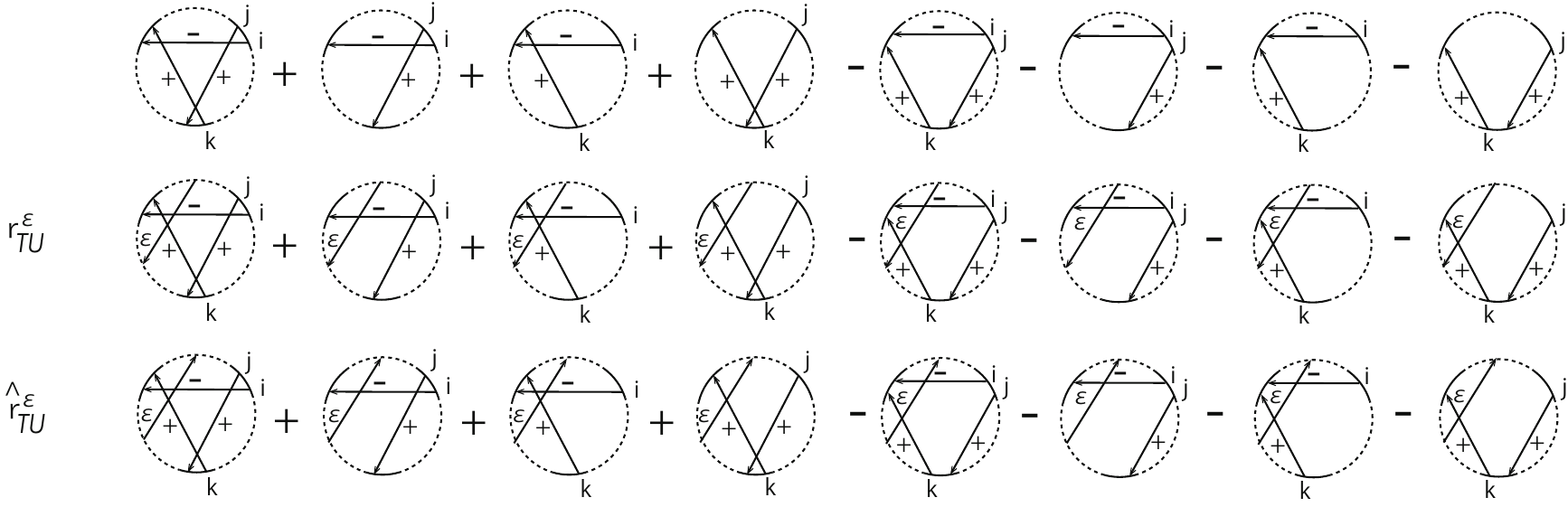}
\caption{The relators of type  ($\sss$) are as in the first line and its mirror image.  The second line is an example $r^{\epsilon}_{TU}$.  $\hat{r}^{\epsilon}_{TU}$ is in the third line  corresponding to ${r}^{\epsilon}_{UT}$.}\label{mir}
\end{figure}

\begin{definition}[truncated reduced Polyak algebra]\label{reducPA}
The \emph{reduced} Polyak algebra is defined as the quotient module $\mathbb{\mathbb{Z}}[\check{G}_{< \infty}] / \reduced(\cR^{\min}_{10110})$. Let $n$ be an integer ($n \ge 2$).   
The \emph{truncated reduced Polyak algebra} is $\mathbb{\mathbb{Z}}[\check{G}_{\le n}] / \reduced(\check{O}_{2, n}(\cR^{\min}_{10110}))$ that is the quotient module of the reduced Polyak algebra and is denoted by $\reduced(\mathcal{P}_n)$.    
\end{definition}
\begin{lemma}\label{lem1}
Let $D_K$ and $D'_K$ be two oriented classical knot diagrams between 
a single strong~$\mathcal{RI\!I\!I}$, hence, there exist two Gauss words $G^*=SijT\bar{k}\,\bar{i}U\bar{j}k$ $(SkjTi\bar{k}U\bar{j}\,\bar{i}$, resp.$)$ and ${{G'}^*}=SjiT\bar{i}\,\bar{k}Uk\bar{j}$ $(SjkT\bar{k}iU\bar{i}\,\bar{j}$, resp.$)$ corresponding to $D_K$ and $D'_K$, respectively, i.e., $AD_{D_K}$ $=$ $[[SijT\bar{k}\,\bar{i}U\bar{j}k]]$ $([[SkjTi\bar{k}U\bar{j}\,\bar{i}]]$, resp.$)$ and $AD_{D'_K}$ $=$ $[[SjiT\bar{i}\,\bar{k}Uk\bar{j}]]$ $([[SjkT\bar{k}iU\bar{i}\,\bar{j}]]$, resp.$)$.  For $L$ $=$ $K$ or $K'$, let $S(L)$ be $S$ of $L$ and $T(L)$ $T$ of $L$.  Then, if there exists an arrow $\lambda$ oriented  from $S(L)$ to $T(L)$ $(S(L)$ to $T(L)$,~resp.$)$, we say that $\lambda \in S(L)$, $\bar{\lambda} \in T(L)$ $(\bar{\lambda} \in S(L)$, $\lambda \in T(L), resp.)$.  
 
Then, 
\[
\sum_{\lambda \in S(L),~\bar{\lambda} \in T(L)} \sign(\lambda) = \sum_{\lambda \in T(L),~ \bar{\lambda} \in S(L)} \sign(\lambda).  
\]
The same statement as that of   the pair ($S$, $T$) holds for ($T$, $U$) or ($S$, $U$).    
\end{lemma}
\begin{proof}
For a crossing, there exist two types of smoothings  to splice the crossing.  One is of type $A^{-1}$ \cite{ItoS} and the other is of type Seifert as in Fig.~\ref{seifert}.    
\begin{figure}[htbp] 
   \centering 
   \includegraphics[width=5cm]{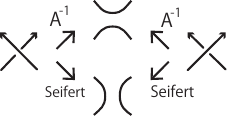} 
   \caption{Seifert splice and $A^{-1}$}
   \label{seifert}
   \end{figure}
For arrows corresponding to letters $i$, $j$, and $k$, we apply splices of type Seifert to $AD_{D_K}$ (type $A^{-1}$ to $AD_{D_{K'}}$,~resp.) as in Fig.~\ref{smm_zu}.   
\begin{figure}[htbp] 
   \centering 
   \includegraphics[width=5cm]{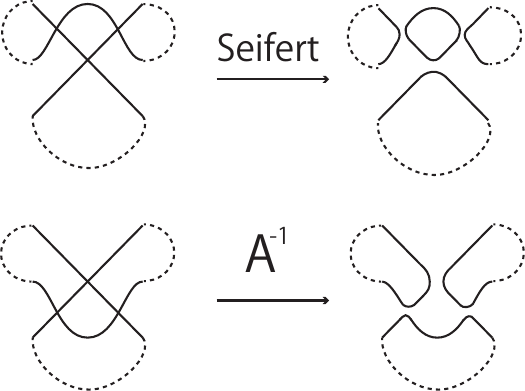} 
   \caption{Applications of splices to $AD_{D_K}$ and $AD_{D_{K'}}$}
   \label{smm_zu} 
\end{figure}
By considering a linking number of two components {\color{black}{(\cite[Theorem~5]{PV}, \cite[Section~4.1]{ostlund}) }} in the three component links, we have the statement.   
\end{proof} 
\begin{proposition}\label{prop2}
Let $\check{G}_{\le 3}$, $\mathbb{Z}[\check{G}_{< \infty}]$, $\check{n}_3$ $=$ $|\check{G}_{\le 3}|$, $\check{G}_{2, 3}$ $=$ $\{ x^*_i \}_{\check{n}_{1} +1 \le i \le \check{n}_{3}}$, and $\check{O}_{2, 3}(\cRep)$ be as in Section~\ref{sec2}.     
Let $\cFsums$ and $\caFsums$ be functions as in Definition~\ref{def_xiv}.  
Suppose that the coefficients $(\alpha_{\check{n}_1 + 1}, \alpha_{\check{n}_1 + 2}, \dots, \alpha_{\check{n}_3})$ satisfy an unsigned rule as in Definition~\ref{remarkGauss} for $\mathbb{Z}[\check{G}_{3, 3}]$ and   
suppose that for every relator $r^*$ having a mirroring pair $(r^* \in \check{O}_{3, 3}(\check{R}_{0 0 0 1 0}))$ and for every relator $R^*$ of the others $(R^* \in \check{R}_{0 0 0 1 0})$, the condition that
\begin{center}
$\caFsums  (r^{\epsilon}_{\cdot \cdot} + \hat{r}^{\epsilon}_{\cdot \cdot}) = 0$ and $\caFsums (R^*) = 0$.   
\end{center}
Then, $\displaystyle \cFsums$
is an integer-valued invariant under the third Reidemeister move of type $\mathcal{RI\!I\!I}$ for  classical knot diagrams.    
\end{proposition}
\begin{proof}
Recall the proof of the case $\epsilon_4 =1$ of Theorem~\ref{gg_thm2} and put  $b=2$ and $d=3$.   

For every $z^*_0 \in \sub^{(0)}(G^*)$, let $\mathcal{R}(z^*_0)$ be a relator in $\cR_{00010}$ obtained by $z^*_0$.  By (\ref{rel_eq}), 
\begin{align*}
&\cFsums (D_K) -  \cFsums (D'_K)  \\
&= \sum_{z^*_0 \in \sub^{(0)}(G^*)} \caFsums  \Big(  ([[z^*_3 (z^*_0)]] + [[z^*_{2a} (z^*_0)]] + [[z^*_{2b} (z^*_0)]] + [[z^*_{2c} (z^*_0)]])\\ 
& \qquad\qquad - ([[{z'}^*_3(z^*_0)]] + [[{z'}^*_{2a}(z^*_0)]] + [[{z'}^*_{2b}(z^*_0)]] + [[{z'}^*_{2c}(z^*_0)]]) \Big). \\
\end{align*}
Here, note that the formula, depending on $z^*_0$, 
\begin{align*}
&[[z^*_3 (z^*_0)]] + [[z^*_{2a} (z^*_0)]] + [[z^*_{2b} (z^*_0)]] + [[z^*_{2c} (z^*_0)]] \\
&- ([[{z'}^*_3(z^*_0)]] + [[{z'}^*_{2a}(z^*_0)]] + [[{z'}^*_{2b}(z^*_0)]] + [[{z'}^*_{2c}(z^*_0)]]) 
\end{align*}
is the form of Type ($\check{\sss}$) relator.  
Then, let $r^*(z^*_0)$ be a projection $O_{2, 3}$ of a relator obtained by $z^*_0$ having a mirroring pair $(r^*(z^*_0) \in \check{O}_{3, 3}(\check{R}_{0 0 0 1 0}))$, let $\hat{r}^*(z^*_0)$ be a (projection of a) relator of a mirroring pair $(r^*(z^*_0), \hat{r}^*(z^*_0))$, and let $R^*(z^*_0)$ be a (projection of a) relator of the others obtained by $z^*_0$ $(R^* \in \check{R}_{0 0 0 1 0})$.  
Then,    
\begin{align*}
&\cFsums (D_K) - \cFsums (D'_K)  \\
&= \caFsums  \left(\sum_{z^*_0 \in \sub^{(0)}(G^*)} r^* (z^*_0) + \sum_{z^*_0 \in \sub^{(0)}(G^*)} \hat{r}^* (z^*_0) + \sum_{z^*_0 \in \sub^{(0)}(G^*)} R^* (z^*_0) \right)   \\  
%
%
&= \sum_{z^*_0 \in \sub^{(0)}(G^*)} \left( \caFsums  ( r^* (z^*_0) + \hat{r}^* (z^*_0) ) +  \caFsums  (R^* (z^*_0)) \right).\\
\end{align*}
%
%
Note that by Lemma~\ref{lem1}, for every $z^*_0 \in \sub^{(0)}(G^*)$ and for every nonzero $r^*(z_0)$, there exists $\hat{r}^*(z_0)$. 
 
Thus, for every $r^* (\in \check{O}_{3, 3}(\check{R}_{0 0 0 1 0}))$ having a mirroring pair ($r^*, \hat{r}^*$) and for every $R^*$ of the others $(R^* \in \check{R}_{0 0 0 1 0})$, 
\begin{center}
$\caFsums  (r^* +  \hat{r}^*) = 0$ and $\caFsums (R^*) = 0$  
\end{center}
implies 
\[
\cFsums (D_K) - \cFsums (D'_K) = 0  
\] 
for a single strong $\mathcal{RI\!I\!I}$ between $D_K$ and $D'_K$.
\end{proof}

If we consider the function of the form $\cFsumIrrs$ of Corollary~\ref{g_thm2} in Proposition~\ref{prop2}, we have:  
\begin{corollary}\label{g_thm2a}

Let $\check{G}_{\le 3}$, $\{ x^*_i \}_{i \in \mathbb{N}}$, $\mathbb{Z}[\check{G}_{< \infty}]$, $\check{n}_3$ $=$ $|\check{G}_{\le 3}|$, $\check{G}_{2, 3}$ $=$ $\{ x^*_i \}_{\check{n}_{1} +1 \le i \le \check{n}_{3}}$, and $\check{O}_{2, 3}(\cRep)$ be as in Section~\ref{sec2}.     
Let $\cFsumIrrs$ and $\caFsumIrrs$ be functions as in Definition~\ref{def_xiv} and Corollary~\ref{g_thm2}. 
Suppose that the coefficients $(\alpha_{\check{n}_1 + 1}, \alpha_{\check{n}_1 + 2},$ $\dots, \alpha_{\check{n}_3})$ satisfy an unsigned rule as in Definition~\ref{remarkGauss} for $\mathbb{Z}[\check{G}_{3, 3}]$ and   
suppose that for every relator $r^*$ having a mirroring pair $(r^* \in \check{O}_{3, 3}(\check{R}_{0 0 0 1 0}))$ and for every relator $R^*$ of the others $(R^* \in \check{R}_{0 0 1 1 0})$, the condition that 
\begin{center}
$  \caFsumIrrs  (r^* +  \hat{r}^*) = 0$ and $\caFsumIrrs (R^*) = 0$.   
\end{center}
Then, $\displaystyle \cFsumIrrs$
is an integer-valued invariant of classical knots.

\end{corollary}

If we consider the function of the form $\cFsumConns$ of Corollary~\ref{cor3b} in Proposition~\ref{prop2}, we have:  
\begin{corollary}\label{cor3ba}
Let $\check{G}_{\le 3}$, $\mathbb{Z}[\check{G}_{< \infty}]$, $\check{n}_3$ $=$ $|\check{G}_{\le 3}|$, $\check{G}_{2, 3}$ $=$ $\{ x^*_i \}_{\check{n}_{1} +1 \le i \le \check{n}_{3}}$, and $\check{O}_{2, 3}(\cRep)$ be as in Section~\ref{sec2}.     
Let $\cFsumConns$ and $\caFsumConns$ be functions as in Definition~\ref{def_xiv} and Corollary~\ref{cor3b}.  
Suppose that the coefficients $(\alpha_{\check{n}_1 + 1}, \alpha_{\check{n}_1 + 2}, \dots, \alpha_{\check{n}_3})$ satisfy an unsigned rule as in Definition~\ref{remarkGauss} for $\mathbb{Z}[\check{G}_{3, 3}]$ and   
suppose that for every relator $r^*$ having a mirroring pair $(r^* \in \check{O}_{3, 3}(\check{R}_{0 0 0 1 0}))$ and for every relator $R^*$ of the others $(R^* \in \check{R}_{0 0 1 1 0})$, the condition that
\begin{center}
$  \caFsumConns  (r^* + \hat{r}^*) = 0$ and $\caFsumConns (R^*) = 0$.   
\end{center}
Then, $\displaystyle \cFsumConns$
is an integer-valued additive invariant of classical knots.

\end{corollary}

\section{Notations}\label{SecNotation}
{\color{black}{
From Sections~\ref{proofThm1} to the end of this paper, in order to describe several proofs, we need to fix orders of arrow diagrams and relators.  First, Notation~\ref{order} fix the order of arrow diagrams.  
\begin{notation}\label{order}
Let $y^{*}_{1} = \input{y1}$,
$y^{*}_{2} = \input{y2}$,
$y^{*}_{3} = \input{y3}$,
$y^{*}_{4} = \input{y4}$,
$y^{*}_{5} = \input{y5}$,
$y^{*}_{6} = \input{y6}$,
$y^{*}_{7} = \input{y7}$,
$y^{*}_{8} = \input{y8}$,
$y^{*}_{9} = \input{y9}$,
$y^{*}_{10} = \input{y10}$,
$y^{*}_{11} = \input{y11}$,
$y^{*}_{12} = \input{y12}$,
$y^{*}_{13} = \input{y13}$,
$y^{*}_{14} = \input{y14}$,
$y^{*}_{15} = \input{y15}$,
$y^{*}_{16} = \input{y16}$,
$y^{*}_{17} = \input{y17}$,
$y^{*}_{18} = \input{y18}$,
$y^{*}_{19} = \input{y19}$,
$y^{*}_{20} = \input{y20}$,
$y^{*}_{21} = \input{y21}$,
$y^{*}_{22} = \input{y22}$,
$y^{*}_{23} = \input{y23}$,
$y^{*}_{24} = \input{y24}$,
$y^{*}_{25} = \input{y25}$,
$y^{*}_{26} = \input{y26}$,
$y^{*}_{27} = \input{y27}$,
$y^{*}_{28} = \input{y28}$,
$y^{*}_{29} = \input{y29}$,
$y^{*}_{30} = \input{y30}$,
$y^{*}_{31} = \input{y31}$,
$y^{*}_{32} = \input{y32}$,
$y^{*}_{33} = \begin{tikzpicture}[baseline=0pt]
\draw (0pt, 0pt) circle (7pt);
\draw(7pt, 0pt)--(-7pt, 0pt);
\fill (-4pt,1.2pt)--(-7pt,0pt)--(-4pt,-1.2pt)--cycle;
\draw(0pt, 7pt)--(0pt, -7pt);
\fill (-1.2pt,-4pt)--(0pt,-7pt)--(1.2pt,-4pt)--cycle;
\fill (4.94975pt, -4.94975pt) circle (1pt);
\draw[font=\tiny] (-10.5pt, 0pt) node {$-$};
\draw[font=\tiny] (0pt, -10.5pt) node {$-$};
\end{tikzpicture}
$,
$y^{*}_{34} = \begin{tikzpicture}[baseline=0pt]
\draw (0pt, 0pt) circle (7pt);
\draw(7pt, 0pt)--(-7pt, 0pt);
\fill (-4pt,1.2pt)--(-7pt,0pt)--(-4pt,-1.2pt)--cycle;
\draw(0pt, 7pt)--(0pt, -7pt);
\fill (-1.2pt,-4pt)--(0pt,-7pt)--(1.2pt,-4pt)--cycle;
\fill (4.94975pt, -4.94975pt) circle (1pt);
\draw[font=\tiny] (-10.5pt, 0pt) node {$-$};
\draw[font=\tiny] (0pt, -10.5pt) node {$+$};
\end{tikzpicture}
$,
$y^{*}_{35} = \begin{tikzpicture}[baseline=0pt]
\draw (0pt, 0pt) circle (7pt);
\draw(7pt, 0pt)--(-7pt, 0pt);
\fill (-4pt,1.2pt)--(-7pt,0pt)--(-4pt,-1.2pt)--cycle;
\draw(0pt, 7pt)--(0pt, -7pt);
\fill (-1.2pt,-4pt)--(0pt,-7pt)--(1.2pt,-4pt)--cycle;
\fill (4.94975pt, -4.94975pt) circle (1pt);
\draw[font=\tiny] (-10.5pt, 0pt) node {$+$};
\draw[font=\tiny] (0pt, -10.5pt) node {$-$};
\end{tikzpicture}
$,
$y^{*}_{36} = \begin{tikzpicture}[baseline=0pt]
\draw (0pt, 0pt) circle (7pt);
\draw(7pt, 0pt)--(-7pt, 0pt);
\fill (-4pt,1.2pt)--(-7pt,0pt)--(-4pt,-1.2pt)--cycle;
\draw(0pt, 7pt)--(0pt, -7pt);
\fill (-1.2pt,-4pt)--(0pt,-7pt)--(1.2pt,-4pt)--cycle;
\fill (4.94975pt, -4.94975pt) circle (1pt);
\draw[font=\tiny] (-10.5pt, 0pt) node {$+$};
\draw[font=\tiny] (0pt, -10.5pt) node {$+$};
\end{tikzpicture}
$,
$y^{*}_{37} = \begin{tikzpicture}[baseline=0pt]
\draw (0pt, 0pt) circle (7pt);
\draw(7pt, 0pt)--(-7pt, 0pt);
\fill (-4pt,1.2pt)--(-7pt,0pt)--(-4pt,-1.2pt)--cycle;
\draw(0pt, 7pt)--(0pt, -7pt);
\fill (-1.2pt,-4pt)--(0pt,-7pt)--(1.2pt,-4pt)--cycle;
\fill (4.94975pt, 4.94975pt) circle (1pt);
\draw[font=\tiny] (-10.5pt, 0pt) node {$-$};
\draw[font=\tiny] (0pt, -10.5pt) node {$-$};
\end{tikzpicture}
$,
$y^{*}_{38} = \begin{tikzpicture}[baseline=0pt]
\draw (0pt, 0pt) circle (7pt);
\draw(7pt, 0pt)--(-7pt, 0pt);
\fill (-4pt,1.2pt)--(-7pt,0pt)--(-4pt,-1.2pt)--cycle;
\draw(0pt, 7pt)--(0pt, -7pt);
\fill (-1.2pt,-4pt)--(0pt,-7pt)--(1.2pt,-4pt)--cycle;
\fill (4.94975pt, 4.94975pt) circle (1pt);
\draw[font=\tiny] (-10.5pt, 0pt) node {$-$};
\draw[font=\tiny] (0pt, -10.5pt) node {$+$};
\end{tikzpicture}
$,
$y^{*}_{39} = \begin{tikzpicture}[baseline=0pt]
\draw (0pt, 0pt) circle (7pt);
\draw(7pt, 0pt)--(-7pt, 0pt);
\fill (-4pt,1.2pt)--(-7pt,0pt)--(-4pt,-1.2pt)--cycle;
\draw(0pt, 7pt)--(0pt, -7pt);
\fill (-1.2pt,-4pt)--(0pt,-7pt)--(1.2pt,-4pt)--cycle;
\fill (4.94975pt, 4.94975pt) circle (1pt);
\draw[font=\tiny] (-10.5pt, 0pt) node {$+$};
\draw[font=\tiny] (0pt, -10.5pt) node {$-$};
\end{tikzpicture}
$,
$y^{*}_{40} = \begin{tikzpicture}[baseline=0pt]
\draw (0pt, 0pt) circle (7pt);
\draw(7pt, 0pt)--(-7pt, 0pt);
\fill (-4pt,1.2pt)--(-7pt,0pt)--(-4pt,-1.2pt)--cycle;
\draw(0pt, 7pt)--(0pt, -7pt);
\fill (-1.2pt,-4pt)--(0pt,-7pt)--(1.2pt,-4pt)--cycle;
\fill (4.94975pt, 4.94975pt) circle (1pt);
\draw[font=\tiny] (-10.5pt, 0pt) node {$+$};
\draw[font=\tiny] (0pt, -10.5pt) node {$+$};
\end{tikzpicture}
$,
$y^{*}_{41} = \begin{tikzpicture}[baseline=0pt]
\draw (0pt, 0pt) circle (7pt);
\draw(7pt, 0pt)--(-7pt, 0pt);
\fill (-4pt,1.2pt)--(-7pt,0pt)--(-4pt,-1.2pt)--cycle;
\draw(0pt, 7pt)--(0pt, -7pt);
\fill (-1.2pt,-4pt)--(0pt,-7pt)--(1.2pt,-4pt)--cycle;
\fill (-4.94975pt, 4.94975pt) circle (1pt);
\draw[font=\tiny] (-10.5pt, 0pt) node {$-$};
\draw[font=\tiny] (0pt, -10.5pt) node {$-$};
\end{tikzpicture}
$,
$y^{*}_{42} = \begin{tikzpicture}[baseline=0pt]
\draw (0pt, 0pt) circle (7pt);
\draw(7pt, 0pt)--(-7pt, 0pt);
\fill (-4pt,1.2pt)--(-7pt,0pt)--(-4pt,-1.2pt)--cycle;
\draw(0pt, 7pt)--(0pt, -7pt);
\fill (-1.2pt,-4pt)--(0pt,-7pt)--(1.2pt,-4pt)--cycle;
\fill (-4.94975pt, 4.94975pt) circle (1pt);
\draw[font=\tiny] (-10.5pt, 0pt) node {$-$};
\draw[font=\tiny] (0pt, -10.5pt) node {$+$};
\end{tikzpicture}
$,
$y^{*}_{43} = \begin{tikzpicture}[baseline=0pt]
\draw (0pt, 0pt) circle (7pt);
\draw(7pt, 0pt)--(-7pt, 0pt);
\fill (-4pt,1.2pt)--(-7pt,0pt)--(-4pt,-1.2pt)--cycle;
\draw(0pt, 7pt)--(0pt, -7pt);
\fill (-1.2pt,-4pt)--(0pt,-7pt)--(1.2pt,-4pt)--cycle;
\fill (-4.94975pt, 4.94975pt) circle (1pt);
\draw[font=\tiny] (-10.5pt, 0pt) node {$+$};
\draw[font=\tiny] (0pt, -10.5pt) node {$-$};
\end{tikzpicture}
$,
$y^{*}_{44} = \begin{tikzpicture}[baseline=0pt]
\draw (0pt, 0pt) circle (7pt);
\draw(7pt, 0pt)--(-7pt, 0pt);
\fill (-4pt,1.2pt)--(-7pt,0pt)--(-4pt,-1.2pt)--cycle;
\draw(0pt, 7pt)--(0pt, -7pt);
\fill (-1.2pt,-4pt)--(0pt,-7pt)--(1.2pt,-4pt)--cycle;
\fill (-4.94975pt, 4.94975pt) circle (1pt);
\draw[font=\tiny] (-10.5pt, 0pt) node {$+$};
\draw[font=\tiny] (0pt, -10.5pt) node {$+$};
\end{tikzpicture}
$,
$y^{*}_{45} = \begin{tikzpicture}[baseline=0pt]
\draw (0pt, 0pt) circle (7pt);
\draw(7pt, 0pt)--(-7pt, 0pt);
\fill (-4pt,1.2pt)--(-7pt,0pt)--(-4pt,-1.2pt)--cycle;
\draw(0pt, 7pt)--(0pt, -7pt);
\fill (-1.2pt,-4pt)--(0pt,-7pt)--(1.2pt,-4pt)--cycle;
\fill (-4.94975pt, -4.94975pt) circle (1pt);
\draw[font=\tiny] (-10.5pt, 0pt) node {$-$};
\draw[font=\tiny] (0pt, -10.5pt) node {$-$};
\end{tikzpicture}
$,
$y^{*}_{46} = \begin{tikzpicture}[baseline=0pt]
\draw (0pt, 0pt) circle (7pt);
\draw(7pt, 0pt)--(-7pt, 0pt);
\fill (-4pt,1.2pt)--(-7pt,0pt)--(-4pt,-1.2pt)--cycle;
\draw(0pt, 7pt)--(0pt, -7pt);
\fill (-1.2pt,-4pt)--(0pt,-7pt)--(1.2pt,-4pt)--cycle;
\fill (-4.94975pt, -4.94975pt) circle (1pt);
\draw[font=\tiny] (-10.5pt, 0pt) node {$-$};
\draw[font=\tiny] (0pt, -10.5pt) node {$+$};
\end{tikzpicture}
$,
$y^{*}_{47} = \begin{tikzpicture}[baseline=0pt]
\draw (0pt, 0pt) circle (7pt);
\draw(7pt, 0pt)--(-7pt, 0pt);
\fill (-4pt,1.2pt)--(-7pt,0pt)--(-4pt,-1.2pt)--cycle;
\draw(0pt, 7pt)--(0pt, -7pt);
\fill (-1.2pt,-4pt)--(0pt,-7pt)--(1.2pt,-4pt)--cycle;
\fill (-4.94975pt, -4.94975pt) circle (1pt);
\draw[font=\tiny] (-10.5pt, 0pt) node {$+$};
\draw[font=\tiny] (0pt, -10.5pt) node {$-$};
\end{tikzpicture}
$,
$y^{*}_{48} = \begin{tikzpicture}[baseline=0pt]
\draw (0pt, 0pt) circle (7pt);
\draw(7pt, 0pt)--(-7pt, 0pt);
\fill (-4pt,1.2pt)--(-7pt,0pt)--(-4pt,-1.2pt)--cycle;
\draw(0pt, 7pt)--(0pt, -7pt);
\fill (-1.2pt,-4pt)--(0pt,-7pt)--(1.2pt,-4pt)--cycle;
\fill (-4.94975pt, -4.94975pt) circle (1pt);
\draw[font=\tiny] (-10.5pt, 0pt) node {$+$};
\draw[font=\tiny] (0pt, -10.5pt) node {$+$};
\end{tikzpicture}
$,

$y^{*}_{49} = \input{y49}$,
$y^{*}_{50} = \input{y50}$,
$y^{*}_{51} = \input{y51}$,
$y^{*}_{52} = \input{y52}$,
$y^{*}_{53} = \input{y53}$,
$y^{*}_{54} = \input{y54}$,
$y^{*}_{55} = \input{y55}$,
$y^{*}_{56} = \input{y56}$,
$y^{*}_{57} = \input{y57}$,
$y^{*}_{58} = \input{y58}$,
$y^{*}_{59} = \input{y59}$,
$y^{*}_{60} = \input{y60}$,
$y^{*}_{61} = \input{y61}$,
$y^{*}_{62} = \input{y62}$,
$y^{*}_{63} = \input{y63}$,
$y^{*}_{64} = \input{y64}$,
$y^{*}_{65} = \input{y65}$,
$y^{*}_{66} = \input{y66}$,
$y^{*}_{67} = \input{y67}$,
$y^{*}_{68} = \input{y68}$,
$y^{*}_{69} = \input{y69}$,
$y^{*}_{70} = \input{y70}$,
$y^{*}_{71} = \input{y71}$,
$y^{*}_{72} = \input{y72}$,
$y^{*}_{73} = \input{y73}$,
$y^{*}_{74} = \input{y74}$,
$y^{*}_{75} = \input{y75}$,
$y^{*}_{76} = \input{y76}$,
$y^{*}_{77} = \input{y77}$,
$y^{*}_{78} = \input{y78}$,
$y^{*}_{79} = \input{y79}$,
$y^{*}_{80} = \input{y80}$,
$y^{*}_{81} = \input{y81}$,
$y^{*}_{82} = \input{y82}$,
$y^{*}_{83} = \input{y83}$,
$y^{*}_{84} = \input{y84}$,
$y^{*}_{85} = \input{y85}$,
$y^{*}_{86} = \input{y86}$,
$y^{*}_{87} = \input{y87}$,
$y^{*}_{88} = \input{y88}$,
$y^{*}_{89} = \input{y89}$,
$y^{*}_{90} = \input{y90}$,
$y^{*}_{91} = \input{y91}$,
$y^{*}_{92} = \input{y92}$,
$y^{*}_{93} = \input{y93}$,
$y^{*}_{94} = \input{y94}$,
$y^{*}_{95} = \input{y95}$,
$y^{*}_{96} = \input{y96}$,
$y^{*}_{97} = \input{y97}$,
$y^{*}_{98} = \input{y98}$,
$y^{*}_{99} = \input{y99}$,
$y^{*}_{100} = \input{y100}$,
$y^{*}_{101} = \input{y101}$,
$y^{*}_{102} = \input{y102}$,
$y^{*}_{103} = \input{y103}$,
$y^{*}_{104} = \input{y104}$,
$y^{*}_{105} = \input{y105}$,
$y^{*}_{106} = \input{y106}$,
$y^{*}_{107} = \input{y107}$,
$y^{*}_{108} = \input{y108}$,
$y^{*}_{109} = \input{y109}$,
$y^{*}_{110} = \input{y110}$,
$y^{*}_{111} = \input{y111}$,
$y^{*}_{112} = \input{y112}$,
$y^{*}_{113} = \input{y113}$,
$y^{*}_{114} = \input{y114}$,
$y^{*}_{115} = \input{y115}$,
$y^{*}_{116} = \input{y116}$,
$y^{*}_{117} = \input{y117}$,
$y^{*}_{118} = \input{y118}$,
$y^{*}_{119} = \input{y119}$,
$y^{*}_{120} = \input{y120}$,
$y^{*}_{121} = \input{y121}$,
$y^{*}_{122} = \input{y122}$,
$y^{*}_{123} = \input{y123}$,
$y^{*}_{124} = \input{y124}$,
$y^{*}_{125} = \input{y125}$,
$y^{*}_{126} = \input{y126}$,
$y^{*}_{127} = \input{y127}$,
$y^{*}_{128} = \input{y128}$,
$y^{*}_{129} = \input{y129}$,
$y^{*}_{130} = \input{y130}$,
$y^{*}_{131} = \input{y131}$,
$y^{*}_{132} = \input{y132}$,
$y^{*}_{133} = \input{y133}$,
$y^{*}_{134} = \input{y134}$,
$y^{*}_{135} = \input{y135}$,
$y^{*}_{136} = \input{y136}$,
$y^{*}_{137} = \input{y137}$,
$y^{*}_{138} = \input{y138}$,
$y^{*}_{139} = \input{y139}$,
$y^{*}_{140} = \input{y140}$,
$y^{*}_{141} = \input{y141}$,
$y^{*}_{142} = \input{y142}$,
$y^{*}_{143} = \input{y143}$,
$y^{*}_{144} = \input{y144}$,
$y^{*}_{145} = \input{y145}$,
$y^{*}_{146} = \input{y146}$,
$y^{*}_{147} = \input{y147}$,
$y^{*}_{148} = \input{y148}$,
$y^{*}_{149} = \input{y149}$,
$y^{*}_{150} = \input{y150}$,
$y^{*}_{151} = \input{y151}$,
$y^{*}_{152} = \input{y152}$,
$y^{*}_{153} = \input{y153}$,
$y^{*}_{154} = \input{y154}$,
$y^{*}_{155} = \input{y155}$,
$y^{*}_{156} = \input{y156}$,
$y^{*}_{157} = \input{y157}$,
$y^{*}_{158} = \input{y158}$,
$y^{*}_{159} = \input{y159}$,
$y^{*}_{160} = \input{y160}$,
$y^{*}_{161} = \input{y161}$,
$y^{*}_{162} = \input{y162}$,
$y^{*}_{163} = \input{y163}$,
$y^{*}_{164} = \input{y164}$,
$y^{*}_{165} = \input{y165}$,
$y^{*}_{166} = \input{y166}$,
$y^{*}_{167} = \input{y167}$, and 
$y^{*}_{168} = \input{y168}$ .  
\end{notation}
Next, we fix the orders of relators we will use.  
\begin{notation}
The notation of relators of type ($\check{\ii}$), type ($\check{\w}$), and type ($\check{\sss}$) obey Definition~\ref{def_relators_arrow} and their orders of relators are given by \cite{TakamuraURL}.   Since we should handle many relators, thus we comment on the number of relators, which affects sizes of matrices appearing proofs of main results.     

\noindent $\bullet$ Type~(\ii): 122  relators $r^*_i$ ($1 \le i \le 122)$, each of which has the form $[[Si\bar{i}T]]$ or $[[S\bar{i}iT]]$, where the length of Gauss word $ST$ is 2 or 4 and $\sign(i)=\pm$.  

\noindent $\bullet$ Type~(\w): 96 relators $r^*_i$ ($123 \le i \le 218$), each of which has the form $[[S\bar{i}\,\bar{j}TijU]]$ $+$ $[[S\bar{i}TiU]]$ $+$ $[[S\bar{j}TjU]]$ or $[[SijT\bar{i}\,\bar{j}U]]$ $+$ $[[SiT\bar{i}U]]$ $+$ $[[SjT\bar{j}U]]$, where the length of Gauss word $STU$ is 0 or 2 and $\sign(i)=\pm$ and $\sign(j)=\mp$.

\noindent $\bullet$ Type~(\sss): 246 relators $r^*_i$ ($219 \le i \le 464$), each of which is a 8-term formula, the length of Gauss word $STUV$ is 0 or 2, including the first term:  
$[[SijT\bar{k}\bar{i}U\bar{j}\,kV]]$, $[[S\bar{k}\bar{i}T\bar{j}\,kUijV]]$, $[[S\bar{j}\,kTijU\bar{k}\bar{i}V]]$,    
$[[SkjTi\bar{k}U\bar{j}\,\bar{i}V]]$, $[[S\bar{j}\,\bar{i}TkjUi\bar{k}V]]$, $[[Si\bar{k}T\bar{j}\,\bar{i}UkjV]]$
where  
$(\sign(i),  \sign(j), \sign(k))$ $=$ $(-1,1, 1)$ for the first three and $(\sign(i),  \sign(j), \sign(k))$ $=$ $(1,-1, 1)$ for the latter.    

\noindent $\bullet$ As described in Definition~\ref{mirror_dfn}, each  mirroring pair is given by the form $r^*$ $+$ $\hat{r}^*$.  For degree at most three, $\sharp\sharp$ type ($\sss$) relators appear and  include $\ast\ast$ mirroring pairs (see \cite{TakamuraURL} for the full list).    In \cite{TakamuraURL}, by definition, every mirroring pair is presented by $16$ term relator.  
\end{notation}
}}

\section{Proposition~\ref{thm1} and its proof}\label{proofThm1}
\begin{proposition}\label{thm1}
The seven (two,~resp.) Gauss diagram formulas
\begin{center}
$\tilde{v}_{3, i}(\cdot) = \langle \tilde{f}_{3, i}, \cdot \rangle  \quad (1 \le i \le 7)$ \quad $(\tilde{v}_{2, i}(\cdot) = \langle \tilde{f}_{2, i}, \cdot \rangle  \quad (i=1, 2),~{\text{resp.}})$
\end{center}
are independent Goussarov-Polyak-Viro finite type invariants of degree three $($two,~resp.$)$ for long virtual knots.  Here, $\tilde{f}_{3, i}$ $(1 \le i \le 7)$ $(\tilde{f}_{2, i}$ $(i=1,2)$,~resp.$)$ are defined by

\begin{align*}
\tilde{f}_{3, 1} & :=  \begin{tikzpicture}[baseline=0pt]
\draw (0pt, 0pt) circle (7pt);
\draw(7pt, 0pt)--(-7pt, 0pt);
\fill (-4pt,1.2pt)--(-7pt,0pt)--(-4pt,-1.2pt)--cycle;
\draw(0pt, 7pt)--(0pt, -7pt);
\fill (-1.2pt,-4pt)--(0pt,-7pt)--(1.2pt,-4pt)--cycle;
\fill (4.94975pt, -4.94975pt) circle (1pt);
\draw[font=\tiny] (-10.5pt, 0pt) node {$-$};
\draw[font=\tiny] (0pt, -10.5pt) node {$-$};
\end{tikzpicture}
 - \begin{tikzpicture}[baseline=0pt]
\draw (0pt, 0pt) circle (7pt);
\draw(7pt, 0pt)--(-7pt, 0pt);
\fill (-4pt,1.2pt)--(-7pt,0pt)--(-4pt,-1.2pt)--cycle;
\draw(0pt, 7pt)--(0pt, -7pt);
\fill (-1.2pt,-4pt)--(0pt,-7pt)--(1.2pt,-4pt)--cycle;
\fill (4.94975pt, -4.94975pt) circle (1pt);
\draw[font=\tiny] (-10.5pt, 0pt) node {$+$};
\draw[font=\tiny] (0pt, -10.5pt) node {$+$};
\end{tikzpicture}
 + \input{y137} + \input{y138} + \input{y140} + \input{y142} - \input{y143} - \input{y148} - \input{y152} - \input{y153} \\
& \phantom{=}
+2 \input{y155} + \input{y156} + \input{y157} + \input{y159} 
  + \input{y160} - \input{y161} +3 \input{y167} +2 \input{y168}, \\
\tilde{f}_{3, 2} & :=  \begin{tikzpicture}[baseline=0pt]
\draw (0pt, 0pt) circle (7pt);
\draw(7pt, 0pt)--(-7pt, 0pt);
\fill (-4pt,1.2pt)--(-7pt,0pt)--(-4pt,-1.2pt)--cycle;
\draw(0pt, 7pt)--(0pt, -7pt);
\fill (-1.2pt,-4pt)--(0pt,-7pt)--(1.2pt,-4pt)--cycle;
\fill (4.94975pt, 4.94975pt) circle (1pt);
\draw[font=\tiny] (-10.5pt, 0pt) node {$-$};
\draw[font=\tiny] (0pt, -10.5pt) node {$+$};
\end{tikzpicture}
 - \begin{tikzpicture}[baseline=0pt]
\draw (0pt, 0pt) circle (7pt);
\draw(7pt, 0pt)--(-7pt, 0pt);
\fill (-4pt,1.2pt)--(-7pt,0pt)--(-4pt,-1.2pt)--cycle;
\draw(0pt, 7pt)--(0pt, -7pt);
\fill (-1.2pt,-4pt)--(0pt,-7pt)--(1.2pt,-4pt)--cycle;
\fill (4.94975pt, 4.94975pt) circle (1pt);
\draw[font=\tiny] (-10.5pt, 0pt) node {$+$};
\draw[font=\tiny] (0pt, -10.5pt) node {$-$};
\end{tikzpicture}
 + \input{y138} - \input{y144} + \input{y154} - \input{y157} + \input{y162} - \input{y163}, \\
\tilde{f}_{3, 3} & :=  \begin{tikzpicture}[baseline=0pt]
\draw (0pt, 0pt) circle (7pt);
\draw(7pt, 0pt)--(-7pt, 0pt);
\fill (-4pt,1.2pt)--(-7pt,0pt)--(-4pt,-1.2pt)--cycle;
\draw(0pt, 7pt)--(0pt, -7pt);
\fill (-1.2pt,-4pt)--(0pt,-7pt)--(1.2pt,-4pt)--cycle;
\fill (4.94975pt, 4.94975pt) circle (1pt);
\draw[font=\tiny] (-10.5pt, 0pt) node {$-$};
\draw[font=\tiny] (0pt, -10.5pt) node {$-$};
\end{tikzpicture}
 + \begin{tikzpicture}[baseline=0pt]
\draw (0pt, 0pt) circle (7pt);
\draw(7pt, 0pt)--(-7pt, 0pt);
\fill (-4pt,1.2pt)--(-7pt,0pt)--(-4pt,-1.2pt)--cycle;
\draw(0pt, 7pt)--(0pt, -7pt);
\fill (-1.2pt,-4pt)--(0pt,-7pt)--(1.2pt,-4pt)--cycle;
\fill (4.94975pt, 4.94975pt) circle (1pt);
\draw[font=\tiny] (-10.5pt, 0pt) node {$+$};
\draw[font=\tiny] (0pt, -10.5pt) node {$-$};
\end{tikzpicture}
 + \begin{tikzpicture}[baseline=0pt]
\draw (0pt, 0pt) circle (7pt);
\draw(7pt, 0pt)--(-7pt, 0pt);
\fill (-4pt,1.2pt)--(-7pt,0pt)--(-4pt,-1.2pt)--cycle;
\draw(0pt, 7pt)--(0pt, -7pt);
\fill (-1.2pt,-4pt)--(0pt,-7pt)--(1.2pt,-4pt)--cycle;
\fill (-4.94975pt, -4.94975pt) circle (1pt);
\draw[font=\tiny] (-10.5pt, 0pt) node {$-$};
\draw[font=\tiny] (0pt, -10.5pt) node {$-$};
\end{tikzpicture}
 + \begin{tikzpicture}[baseline=0pt]
\draw (0pt, 0pt) circle (7pt);
\draw(7pt, 0pt)--(-7pt, 0pt);
\fill (-4pt,1.2pt)--(-7pt,0pt)--(-4pt,-1.2pt)--cycle;
\draw(0pt, 7pt)--(0pt, -7pt);
\fill (-1.2pt,-4pt)--(0pt,-7pt)--(1.2pt,-4pt)--cycle;
\fill (-4.94975pt, -4.94975pt) circle (1pt);
\draw[font=\tiny] (-10.5pt, 0pt) node {$+$};
\draw[font=\tiny] (0pt, -10.5pt) node {$-$};
\end{tikzpicture}
 - \input{y138} - \input{y140} - \input{y151} - \input{y154} - \input{y155} - \input{y159} 
\\
& \phantom{=}
- \input{y162} - \input{y165}   -\input{y167} - \input{y168}, \\
\tilde{f}_{3, 4} & := - \begin{tikzpicture}[baseline=0pt]
\draw (0pt, 0pt) circle (7pt);
\draw(7pt, 0pt)--(-7pt, 0pt);
\fill (-4pt,1.2pt)--(-7pt,0pt)--(-4pt,-1.2pt)--cycle;
\draw(0pt, 7pt)--(0pt, -7pt);
\fill (-1.2pt,-4pt)--(0pt,-7pt)--(1.2pt,-4pt)--cycle;
\fill (-4.94975pt, -4.94975pt) circle (1pt);
\draw[font=\tiny] (-10.5pt, 0pt) node {$-$};
\draw[font=\tiny] (0pt, -10.5pt) node {$+$};
\end{tikzpicture}
 + \begin{tikzpicture}[baseline=0pt]
\draw (0pt, 0pt) circle (7pt);
\draw(7pt, 0pt)--(-7pt, 0pt);
\fill (-4pt,1.2pt)--(-7pt,0pt)--(-4pt,-1.2pt)--cycle;
\draw(0pt, 7pt)--(0pt, -7pt);
\fill (-1.2pt,-4pt)--(0pt,-7pt)--(1.2pt,-4pt)--cycle;
\fill (-4.94975pt, -4.94975pt) circle (1pt);
\draw[font=\tiny] (-10.5pt, 0pt) node {$+$};
\draw[font=\tiny] (0pt, -10.5pt) node {$-$};
\end{tikzpicture}
 + \input{y142} + \input{y147} - \input{y151} - \input{y159} - \input{y165} + \input{y166}, \\
\tilde{f}_{3, 5} & :=  \begin{tikzpicture}[baseline=0pt]
\draw (0pt, 0pt) circle (7pt);
\draw(7pt, 0pt)--(-7pt, 0pt);
\fill (-4pt,1.2pt)--(-7pt,0pt)--(-4pt,-1.2pt)--cycle;
\draw(0pt, 7pt)--(0pt, -7pt);
\fill (-1.2pt,-4pt)--(0pt,-7pt)--(1.2pt,-4pt)--cycle;
\fill (-4.94975pt, 4.94975pt) circle (1pt);
\draw[font=\tiny] (-10.5pt, 0pt) node {$-$};
\draw[font=\tiny] (0pt, -10.5pt) node {$-$};
\end{tikzpicture}
 - \begin{tikzpicture}[baseline=0pt]
\draw (0pt, 0pt) circle (7pt);
\draw(7pt, 0pt)--(-7pt, 0pt);
\fill (-4pt,1.2pt)--(-7pt,0pt)--(-4pt,-1.2pt)--cycle;
\draw(0pt, 7pt)--(0pt, -7pt);
\fill (-1.2pt,-4pt)--(0pt,-7pt)--(1.2pt,-4pt)--cycle;
\fill (-4.94975pt, 4.94975pt) circle (1pt);
\draw[font=\tiny] (-10.5pt, 0pt) node {$+$};
\draw[font=\tiny] (0pt, -10.5pt) node {$+$};
\end{tikzpicture}
 + \input{y138} + \input{y139} +2 \input{y140} + \input{y141} + \input{y142} - \input{y145} - \input{y146} - \input{y149} 
\\
& \phantom{=}
 - \input{y150} + \input{y155} + \input{y157}  + \input{y158} + \input{y159} - \input{y164} +2 \input{y167} +3 \input{y168}, \\
\tilde{f}_{3, 6} & := 4 \begin{tikzpicture}[baseline=0pt]
\draw (0pt, 0pt) circle (7pt);
\draw(7pt, 0pt)--(-7pt, 0pt);
\fill (-4pt,1.2pt)--(-7pt,0pt)--(-4pt,-1.2pt)--cycle;
\draw(0pt, 7pt)--(0pt, -7pt);
\fill (-1.2pt,-4pt)--(0pt,-7pt)--(1.2pt,-4pt)--cycle;
\fill (-4.94975pt, -4.94975pt) circle (1pt);
\draw[font=\tiny] (-10.5pt, 0pt) node {$-$};
\draw[font=\tiny] (0pt, -10.5pt) node {$-$};
\end{tikzpicture}
 +4 \begin{tikzpicture}[baseline=0pt]
\draw (0pt, 0pt) circle (7pt);
\draw(7pt, 0pt)--(-7pt, 0pt);
\fill (-4pt,1.2pt)--(-7pt,0pt)--(-4pt,-1.2pt)--cycle;
\draw(0pt, 7pt)--(0pt, -7pt);
\fill (-1.2pt,-4pt)--(0pt,-7pt)--(1.2pt,-4pt)--cycle;
\fill (-4.94975pt, -4.94975pt) circle (1pt);
\draw[font=\tiny] (-10.5pt, 0pt) node {$+$};
\draw[font=\tiny] (0pt, -10.5pt) node {$-$};
\end{tikzpicture}
 + \input{y96} - \input{y102} - \input{y138} -2 \input{y140} + \input{y142} + \input{y143} -4 \input{y151} - \input{y152} \\
 & \phantom{=}
 -2 \input{y155} - \input{y157}  -3 \input{y159} -4 \input{y165} -2 \input{y167} -2 \input{y168}, \\
\tilde{f}_{3, 7} & :=  \input{y72} - \input{y78} - \input{y96} + \input{y102} - \input{y143} - \input{y146} + \input{y149} + \input{y152}, \\
\tilde{f}_{2, 1} & :=  \begin{tikzpicture}[baseline=0pt]
\draw (0pt, 0pt) circle (7pt);
\draw(7pt, 0pt)--(-7pt, 0pt);
\fill (-4pt,1.2pt)--(-7pt,0pt)--(-4pt,-1.2pt)--cycle;
\draw(0pt, 7pt)--(0pt, -7pt);
\fill (-1.2pt,-4pt)--(0pt,-7pt)--(1.2pt,-4pt)--cycle;
\fill (-4.94975pt, -4.94975pt) circle (1pt);
\draw[font=\tiny] (-10.5pt, 0pt) node {$-$};
\draw[font=\tiny] (0pt, -10.5pt) node {$-$};
\end{tikzpicture}
 + \begin{tikzpicture}[baseline=0pt]
\draw (0pt, 0pt) circle (7pt);
\draw(7pt, 0pt)--(-7pt, 0pt);
\fill (-4pt,1.2pt)--(-7pt,0pt)--(-4pt,-1.2pt)--cycle;
\draw(0pt, 7pt)--(0pt, -7pt);
\fill (-1.2pt,-4pt)--(0pt,-7pt)--(1.2pt,-4pt)--cycle;
\fill (-4.94975pt, -4.94975pt) circle (1pt);
\draw[font=\tiny] (-10.5pt, 0pt) node {$-$};
\draw[font=\tiny] (0pt, -10.5pt) node {$+$};
\end{tikzpicture}
 + \begin{tikzpicture}[baseline=0pt]
\draw (0pt, 0pt) circle (7pt);
\draw(7pt, 0pt)--(-7pt, 0pt);
\fill (-4pt,1.2pt)--(-7pt,0pt)--(-4pt,-1.2pt)--cycle;
\draw(0pt, 7pt)--(0pt, -7pt);
\fill (-1.2pt,-4pt)--(0pt,-7pt)--(1.2pt,-4pt)--cycle;
\fill (-4.94975pt, -4.94975pt) circle (1pt);
\draw[font=\tiny] (-10.5pt, 0pt) node {$+$};
\draw[font=\tiny] (0pt, -10.5pt) node {$-$};
\end{tikzpicture}
 + \begin{tikzpicture}[baseline=0pt]
\draw (0pt, 0pt) circle (7pt);
\draw(7pt, 0pt)--(-7pt, 0pt);
\fill (-4pt,1.2pt)--(-7pt,0pt)--(-4pt,-1.2pt)--cycle;
\draw(0pt, 7pt)--(0pt, -7pt);
\fill (-1.2pt,-4pt)--(0pt,-7pt)--(1.2pt,-4pt)--cycle;
\fill (-4.94975pt, -4.94975pt) circle (1pt);
\draw[font=\tiny] (-10.5pt, 0pt) node {$+$};
\draw[font=\tiny] (0pt, -10.5pt) node {$+$};
\end{tikzpicture}
,~{\textrm{and}}~ \\
\tilde{f}_{2, 2} & :=  \begin{tikzpicture}[baseline=0pt]
\draw (0pt, 0pt) circle (7pt);
\draw(7pt, 0pt)--(-7pt, 0pt);
\fill (-4pt,1.2pt)--(-7pt,0pt)--(-4pt,-1.2pt)--cycle;
\draw(0pt, 7pt)--(0pt, -7pt);
\fill (-1.2pt,-4pt)--(0pt,-7pt)--(1.2pt,-4pt)--cycle;
\fill (4.94975pt, 4.94975pt) circle (1pt);
\draw[font=\tiny] (-10.5pt, 0pt) node {$-$};
\draw[font=\tiny] (0pt, -10.5pt) node {$-$};
\end{tikzpicture}
 + \begin{tikzpicture}[baseline=0pt]
\draw (0pt, 0pt) circle (7pt);
\draw(7pt, 0pt)--(-7pt, 0pt);
\fill (-4pt,1.2pt)--(-7pt,0pt)--(-4pt,-1.2pt)--cycle;
\draw(0pt, 7pt)--(0pt, -7pt);
\fill (-1.2pt,-4pt)--(0pt,-7pt)--(1.2pt,-4pt)--cycle;
\fill (4.94975pt, 4.94975pt) circle (1pt);
\draw[font=\tiny] (-10.5pt, 0pt) node {$-$};
\draw[font=\tiny] (0pt, -10.5pt) node {$+$};
\end{tikzpicture}
 + \begin{tikzpicture}[baseline=0pt]
\draw (0pt, 0pt) circle (7pt);
\draw(7pt, 0pt)--(-7pt, 0pt);
\fill (-4pt,1.2pt)--(-7pt,0pt)--(-4pt,-1.2pt)--cycle;
\draw(0pt, 7pt)--(0pt, -7pt);
\fill (-1.2pt,-4pt)--(0pt,-7pt)--(1.2pt,-4pt)--cycle;
\fill (4.94975pt, 4.94975pt) circle (1pt);
\draw[font=\tiny] (-10.5pt, 0pt) node {$+$};
\draw[font=\tiny] (0pt, -10.5pt) node {$-$};
\end{tikzpicture}
 + \begin{tikzpicture}[baseline=0pt]
\draw (0pt, 0pt) circle (7pt);
\draw(7pt, 0pt)--(-7pt, 0pt);
\fill (-4pt,1.2pt)--(-7pt,0pt)--(-4pt,-1.2pt)--cycle;
\draw(0pt, 7pt)--(0pt, -7pt);
\fill (-1.2pt,-4pt)--(0pt,-7pt)--(1.2pt,-4pt)--cycle;
\fill (4.94975pt, 4.94975pt) circle (1pt);
\draw[font=\tiny] (-10.5pt, 0pt) node {$+$};
\draw[font=\tiny] (0pt, -10.5pt) node {$+$};
\end{tikzpicture}
 .  
\end{align*}

Further, the restriction of $\tilde{v}_{3,i}$ $(1 \le i \le 7)$ to classical knots is a Vassiliev knot  invariant of degree three. More precisely, on classical knots $\tilde{v}_{3,1}=\tilde{v}_{3,5}=2 v_3$, $\tilde{v}_{3,2}=\tilde{v}_{3,4}=\tilde{v}_{3,7}=0$, $\tilde{v}_{3,3}=-2 v_3+v_2$ and $\tilde{v}_{3,6}=-4 v_3+2v_2$, where $v_3$ is the Vassiliev knot invariant of degree three which takes values $0$ on the unknot, $+1$ on the right trefoil, and $-1$ on the left trefoil.  
Moreover it is known that on classical knots $\tilde{v}_{2,1}=\tilde{v}_{2,2}=v_2$, where $v_2$ is the Vassiliev knot invariant of degree two which takes values $0$ on the unknot and $+1$ on both the right and left trefoil. 
\end{proposition}
{\color{black}{
\begin{remark}
When we only used connected arrow diagrams (Definition~\ref{defconnected}) simply, the first and second authors  had $\tilde{v}_i$ ($1 \le i \le 5$) whereas two additional explicit  formulas $\tilde{v}_6, \tilde{v}_7$ had been given using unconnected arrow diagrams via computer aid by the third author.  
We also mention that  \cite{ItoTakamura} includes a similar computer computation for different situations.           
 For each Gauss diagram having three arrows, Proposition~\ref{thm1} shows that it is possible to choose representations using unsigned arrow diagrams only for the above five formulas $\tilde{v}_i$ ($1 \le i \le 5$) consisting of connected arrow diagrams.  It is unknown that there is a relation between the connectedness and the representation of invariants by unsigned arrow diagrams.       
\end{remark}
}}

\begin{proof}[Proof of Proposition~\ref{thm1} (1st part)]
{\color{black}{Let $y^*_i$ be an arrow diagram as in Notation~\ref{order}.  

\noindent(Invariance) Theorem~\ref{gg_thm2} implies that if there exist integers $\alpha_i$ ($1 \le i \le 168$) such that 
$\sum_{i=1}^{168} \alpha_i \tilde{y}^*_i (r^*)$ $=$ $0$ ($\forall r^* \in \check{O}_{2, 3}(\check{R}^{\min}_{10110})$), then $\sum_{i=1}^{168} \alpha_i \tilde{y}^*_i$ is a (long) virtual knot invariant.  
It is elementary to show that $\check{O}_{2, 3}(\check{R}^{\min}_{10110})$ consists of  elements $r^*_j$ ($1 \le j \le 464$).  Here, note that $122$ ($96$, $246$, resp.) relators are of type ($\check{\ii}$) (($\check{\w}$), ($\check{\sss}$), resp.).       
Let $\bf{x}$ $=$ $(y^*_1, y^*_2, \ldots, y^*_{168})$.  
Then, by solving the linear equation ${\bf{x}} M$ $=$ $\bf{0}$, we have  such $\alpha_i$ ($1 \le i \le 168$), where $M$ $=$ $(\tilde{y}^*_i (r^*_j))_{1\le i \le 168, 1 \le j \le 464}$, which implies nine (seven plus two) linearly independent functions.     
}}
  
\noindent(Fixing degrees) They are of degree at most  three since it is known that an invariant presented by a Gauss diagram formula is finite type of degree at most the number of arrows in an arrow diagram having the maximal number of arrows among the arrow diagrams in the Gauss diagram formula \cite{gpv}.  
By using virtual trefoil with a base point (Fig.~\ref{VT}), we also check that $\tilde{v}_{3, i}$ ($i=2, 4, 7$) is not degree at most two.   As it is directly shown in the 2nd part below, since the restriction of $\tilde{v}_{3, i}$ ($i=1, 3, 5, 6$) includes the term $v_3$ ($=$ the classical Vassiliev knot invariant of degree three), which detects a trefoil and its mirror image, the claim holds.  

\noindent(Independency) Any two invariants among $\tilde{v}_{3, i}$ ($1 \le i \le 7$) and $\tilde{v}_{2, i}$ ($i=1, 2$) are mutually   linearly independent.  It is clear that this fact implies that each invariant $\tilde{v}_{3, i}$ is independent of the other invariants. In fact, the long virtual knot 
$\atrad$ ($\atraf$, $\atrae$, $\atrac$, or $\atraa$,
 resp.) is distinguished with the trivial virtual long knot by $ \tilde{v}_{3,i}$ ($1 \le i \le 5$ resp.) and not distinguished with the trivial one by the other four invariants. Therefore, we have that 
 $\atrad$, $\atraf$, $\atrae$, $\atrac$, and $\atraa$ are all different. 
By the same argument as the above, for $\tilde{v}_{3, 6}$ and $\tilde{v}_{3, 7}$, we find non-connected arrow diagrams, which implies that they are independent of $\tilde{v}_{3, i}$ ($1 \le i \le 5$) and  both invariants are nontrivial.    Further, it is clear that $\tilde{v}_{3, 7}$ includes a non-connected arrow diagrams that is different from each of non-connected arrow diagrams of $\tilde{v}_{3, 7}$.  Then $\tilde{v}_{3, 7}$ is independent from $\tilde{v}_{3, 6}$.    
\end{proof}

\begin{proof}[Proof of Proposition~\ref{thm1} (2nd part)]
It is known that for any Goussarov-Polyak-Viro finite type invariant of degree $n$, its restriction to classical knots is a Vassiliev knot  invariant of degree $\le n$ \cite{gpv}.
Therefore each $\tilde{v}_{3,i}$ is a Vassiliev knot invariant of degree $\le 3$ for classical knots. 
Here, we set that the Vassiliev knot invariant $v_3$ of degree $3$ takes value $0$ on the unknot, $+1$ on the right trefoil and $-1$ on the left trefoil, and the Vassiliev knot invariant $v_2$ of degree $2$ takes value $0$ on the unknot, $+1$ on the right trefoil and left trefoil. There is no nontrivial Vassiliev knot invariant of degree $1$.
Moreover, the value of the unknot by $\tilde{v}_{3,i}$ vanishes.
Therefore $\tilde{v}_{3,i}$ is a linear sum of $v_3$ and $v_2$.
By calculating the value of the right trefoil and the left trefoil, we have the statement.   
\end{proof}

\section{Proposition~\ref{thm2} and its proof}\label{proofThm2} 
\begin{proposition}\label{thm2}
The twenty-one Gauss diagram formulas 
\begin{center}
$v_{3, i}(\cdot) = \langle f_{3, i}, \cdot \rangle  \quad (1 \le i \le 21)$ and  $\tilde{v}_{2, i}(\cdot) = \langle \tilde{f}_{2, i}, \cdot \rangle \quad (i=1, 2)$
\end{center}
that are classical knot invariants are explicitly given as follows.  Here, $f_{3, i}$ $(1 \le i \le 21)$ and $f_{2, i}$ $(i=1,2)$ are defined by

\begin{align*}
f_{3, 1} & :=  \begin{tikzpicture}[baseline=0pt]
\draw (0pt, 0pt) circle (7pt);
\draw(7pt, 0pt)--(-7pt, 0pt);
\fill (-4pt,1.2pt)--(-7pt,0pt)--(-4pt,-1.2pt)--cycle;
\draw(0pt, 7pt)--(0pt, -7pt);
\fill (-1.2pt,-4pt)--(0pt,-7pt)--(1.2pt,-4pt)--cycle;
\fill (4.94975pt, -4.94975pt) circle (1pt);
\draw[font=\tiny] (-10.5pt, 0pt) node {$-$};
\draw[font=\tiny] (0pt, -10.5pt) node {$-$};
\end{tikzpicture}
 - \begin{tikzpicture}[baseline=0pt]
\draw (0pt, 0pt) circle (7pt);
\draw(7pt, 0pt)--(-7pt, 0pt);
\fill (-4pt,1.2pt)--(-7pt,0pt)--(-4pt,-1.2pt)--cycle;
\draw(0pt, 7pt)--(0pt, -7pt);
\fill (-1.2pt,-4pt)--(0pt,-7pt)--(1.2pt,-4pt)--cycle;
\fill (4.94975pt, -4.94975pt) circle (1pt);
\draw[font=\tiny] (-10.5pt, 0pt) node {$+$};
\draw[font=\tiny] (0pt, -10.5pt) node {$+$};
\end{tikzpicture}
 - \input{y147} - \input{y148} - \input{y155} - \input{y161} - \input{y168}, \\
f_{3, 2} & :=  \begin{tikzpicture}[baseline=0pt]
\draw (0pt, 0pt) circle (7pt);
\draw(7pt, 0pt)--(-7pt, 0pt);
\fill (-4pt,1.2pt)--(-7pt,0pt)--(-4pt,-1.2pt)--cycle;
\draw(0pt, 7pt)--(0pt, -7pt);
\fill (-1.2pt,-4pt)--(0pt,-7pt)--(1.2pt,-4pt)--cycle;
\fill (4.94975pt, 4.94975pt) circle (1pt);
\draw[font=\tiny] (-10.5pt, 0pt) node {$-$};
\draw[font=\tiny] (0pt, -10.5pt) node {$-$};
\end{tikzpicture}
 - \begin{tikzpicture}[baseline=0pt]
\draw (0pt, 0pt) circle (7pt);
\draw(7pt, 0pt)--(-7pt, 0pt);
\fill (-4pt,1.2pt)--(-7pt,0pt)--(-4pt,-1.2pt)--cycle;
\draw(0pt, 7pt)--(0pt, -7pt);
\fill (-1.2pt,-4pt)--(0pt,-7pt)--(1.2pt,-4pt)--cycle;
\fill (4.94975pt, 4.94975pt) circle (1pt);
\draw[font=\tiny] (-10.5pt, 0pt) node {$+$};
\draw[font=\tiny] (0pt, -10.5pt) node {$+$};
\end{tikzpicture}
 - \input{y143} - \input{y144} - \input{y147} - \input{y148} + \input{y153} - \input{y162} - \input{y163}, \\
f_{3, 3} & :=  \begin{tikzpicture}[baseline=0pt]
\draw (0pt, 0pt) circle (7pt);
\draw(7pt, 0pt)--(-7pt, 0pt);
\fill (-4pt,1.2pt)--(-7pt,0pt)--(-4pt,-1.2pt)--cycle;
\draw(0pt, 7pt)--(0pt, -7pt);
\fill (-1.2pt,-4pt)--(0pt,-7pt)--(1.2pt,-4pt)--cycle;
\fill (4.94975pt, 4.94975pt) circle (1pt);
\draw[font=\tiny] (-10.5pt, 0pt) node {$-$};
\draw[font=\tiny] (0pt, -10.5pt) node {$+$};
\end{tikzpicture}
 - \begin{tikzpicture}[baseline=0pt]
\draw (0pt, 0pt) circle (7pt);
\draw(7pt, 0pt)--(-7pt, 0pt);
\fill (-4pt,1.2pt)--(-7pt,0pt)--(-4pt,-1.2pt)--cycle;
\draw(0pt, 7pt)--(0pt, -7pt);
\fill (-1.2pt,-4pt)--(0pt,-7pt)--(1.2pt,-4pt)--cycle;
\fill (4.94975pt, 4.94975pt) circle (1pt);
\draw[font=\tiny] (-10.5pt, 0pt) node {$+$};
\draw[font=\tiny] (0pt, -10.5pt) node {$-$};
\end{tikzpicture}
 + \input{y138} - \input{y144} + \input{y147} + \input{y148} - \input{y153} - \input{y157} + \input{y162} - \input{y163}, \\
f_{3, 4} & :=  \begin{tikzpicture}[baseline=0pt]
\draw (0pt, 0pt) circle (7pt);
\draw(7pt, 0pt)--(-7pt, 0pt);
\fill (-4pt,1.2pt)--(-7pt,0pt)--(-4pt,-1.2pt)--cycle;
\draw(0pt, 7pt)--(0pt, -7pt);
\fill (-1.2pt,-4pt)--(0pt,-7pt)--(1.2pt,-4pt)--cycle;
\fill (-4.94975pt, 4.94975pt) circle (1pt);
\draw[font=\tiny] (-10.5pt, 0pt) node {$-$};
\draw[font=\tiny] (0pt, -10.5pt) node {$-$};
\end{tikzpicture}
 - \begin{tikzpicture}[baseline=0pt]
\draw (0pt, 0pt) circle (7pt);
\draw(7pt, 0pt)--(-7pt, 0pt);
\fill (-4pt,1.2pt)--(-7pt,0pt)--(-4pt,-1.2pt)--cycle;
\draw(0pt, 7pt)--(0pt, -7pt);
\fill (-1.2pt,-4pt)--(0pt,-7pt)--(1.2pt,-4pt)--cycle;
\fill (-4.94975pt, 4.94975pt) circle (1pt);
\draw[font=\tiny] (-10.5pt, 0pt) node {$+$};
\draw[font=\tiny] (0pt, -10.5pt) node {$+$};
\end{tikzpicture}
 + \input{y138} + \input{y139} - \input{y144} - \input{y145} - \input{y164} + \input{y168}, \\
f_{3, 5} & :=  \begin{tikzpicture}[baseline=0pt]
\draw (0pt, 0pt) circle (7pt);
\draw(7pt, 0pt)--(-7pt, 0pt);
\fill (-4pt,1.2pt)--(-7pt,0pt)--(-4pt,-1.2pt)--cycle;
\draw(0pt, 7pt)--(0pt, -7pt);
\fill (-1.2pt,-4pt)--(0pt,-7pt)--(1.2pt,-4pt)--cycle;
\fill (-4.94975pt, -4.94975pt) circle (1pt);
\draw[font=\tiny] (-10.5pt, 0pt) node {$-$};
\draw[font=\tiny] (0pt, -10.5pt) node {$-$};
\end{tikzpicture}
 - \begin{tikzpicture}[baseline=0pt]
\draw (0pt, 0pt) circle (7pt);
\draw(7pt, 0pt)--(-7pt, 0pt);
\fill (-4pt,1.2pt)--(-7pt,0pt)--(-4pt,-1.2pt)--cycle;
\draw(0pt, 7pt)--(0pt, -7pt);
\fill (-1.2pt,-4pt)--(0pt,-7pt)--(1.2pt,-4pt)--cycle;
\fill (-4.94975pt, -4.94975pt) circle (1pt);
\draw[font=\tiny] (-10.5pt, 0pt) node {$+$};
\draw[font=\tiny] (0pt, -10.5pt) node {$+$};
\end{tikzpicture}
 - \input{y138} + \input{y141} - \input{y145} - \input{y146} - \input{y147} - \input{y165} - \input{y166}, \\
f_{3, 6} & :=  \begin{tikzpicture}[baseline=0pt]
\draw (0pt, 0pt) circle (7pt);
\draw(7pt, 0pt)--(-7pt, 0pt);
\fill (-4pt,1.2pt)--(-7pt,0pt)--(-4pt,-1.2pt)--cycle;
\draw(0pt, 7pt)--(0pt, -7pt);
\fill (-1.2pt,-4pt)--(0pt,-7pt)--(1.2pt,-4pt)--cycle;
\fill (-4.94975pt, -4.94975pt) circle (1pt);
\draw[font=\tiny] (-10.5pt, 0pt) node {$-$};
\draw[font=\tiny] (0pt, -10.5pt) node {$+$};
\end{tikzpicture}
 - \begin{tikzpicture}[baseline=0pt]
\draw (0pt, 0pt) circle (7pt);
\draw(7pt, 0pt)--(-7pt, 0pt);
\fill (-4pt,1.2pt)--(-7pt,0pt)--(-4pt,-1.2pt)--cycle;
\draw(0pt, 7pt)--(0pt, -7pt);
\fill (-1.2pt,-4pt)--(0pt,-7pt)--(1.2pt,-4pt)--cycle;
\fill (-4.94975pt, -4.94975pt) circle (1pt);
\draw[font=\tiny] (-10.5pt, 0pt) node {$+$};
\draw[font=\tiny] (0pt, -10.5pt) node {$-$};
\end{tikzpicture}
 - \input{y139} + \input{y145} - \input{y153} + \input{y156} + \input{y165} - \input{y166}, \\
f_{3, 7} & :=  \input{y155} + \input{y156} + \input{y157} + \input{y158} + \input{y159} + \input{y160} +2 \input{y167} +2 \input{y168}, \\
f_{3, 8} & := - \input{y137} - \input{y138} - \input{y139} - \input{y140} - \input{y141} - \input{y142} + \input{y155} + \input{y156} + \input{y157} + \input{y158} + \input{y159} + \input{y160}, \\
f_{3, 9} & := - \input{y143} - \input{y144} - \input{y145} - \input{y146} - \input{y147} - \input{y148} + \input{y149} + \input{y150} + \input{y151} + \input{y152} + \input{y153} + \input{y154}, \\
f_{3, 10} & :=  \input{y137} + \input{y138} + \input{y139} - \input{y155} - \input{y156} - \input{y157}, \\
f_{3, 11} & :=  \input{y140} - \input{y143} - \input{y145} - \input{y146} + \input{y147} + \input{y151} - \input{y153} + \input{y155}, \\
f_{3, 12} & := - \input{y155} + \input{y158}, \\
f_{3, 13} & := - \input{y147} + \input{y153} - \input{y156} + \input{y159}, \\
f_{3, 14} & := - \input{y146} + \input{y152}, \\
f_{3, 15} & :=  \input{y138} - \input{y141} + \input{y145} - \input{y151}, \\
f_{3, 16} & := - \input{y144} - \input{y145} + \input{y150} + \input{y151}, \\
f_{3, 17} & := - \input{y143} + \input{y149}, \\
f_{3, 18} & :=  \input{y139} - \input{y142} - \input{y145} + \input{y151}, \\
f_{3, 19} & :=  \input{y137} + \input{y138} + \input{y139} - \input{y140} - \input{y141} - \input{y142}, \\
f_{3, 20} & :=  \input{y96} - \input{y102}, \\
f_{3, 21} & :=  \input{y72} - \input{y78}, \\
f_{2, 1} & :=  \begin{tikzpicture}[baseline=0pt]
\draw (0pt, 0pt) circle (7pt);
\draw(7pt, 0pt)--(-7pt, 0pt);
\fill (-4pt,1.2pt)--(-7pt,0pt)--(-4pt,-1.2pt)--cycle;
\draw(0pt, 7pt)--(0pt, -7pt);
\fill (-1.2pt,-4pt)--(0pt,-7pt)--(1.2pt,-4pt)--cycle;
\fill (-4.94975pt, -4.94975pt) circle (1pt);
\draw[font=\tiny] (-10.5pt, 0pt) node {$-$};
\draw[font=\tiny] (0pt, -10.5pt) node {$-$};
\end{tikzpicture}
 + \begin{tikzpicture}[baseline=0pt]
\draw (0pt, 0pt) circle (7pt);
\draw(7pt, 0pt)--(-7pt, 0pt);
\fill (-4pt,1.2pt)--(-7pt,0pt)--(-4pt,-1.2pt)--cycle;
\draw(0pt, 7pt)--(0pt, -7pt);
\fill (-1.2pt,-4pt)--(0pt,-7pt)--(1.2pt,-4pt)--cycle;
\fill (-4.94975pt, -4.94975pt) circle (1pt);
\draw[font=\tiny] (-10.5pt, 0pt) node {$-$};
\draw[font=\tiny] (0pt, -10.5pt) node {$+$};
\end{tikzpicture}
 + \begin{tikzpicture}[baseline=0pt]
\draw (0pt, 0pt) circle (7pt);
\draw(7pt, 0pt)--(-7pt, 0pt);
\fill (-4pt,1.2pt)--(-7pt,0pt)--(-4pt,-1.2pt)--cycle;
\draw(0pt, 7pt)--(0pt, -7pt);
\fill (-1.2pt,-4pt)--(0pt,-7pt)--(1.2pt,-4pt)--cycle;
\fill (-4.94975pt, -4.94975pt) circle (1pt);
\draw[font=\tiny] (-10.5pt, 0pt) node {$+$};
\draw[font=\tiny] (0pt, -10.5pt) node {$-$};
\end{tikzpicture}
 + \begin{tikzpicture}[baseline=0pt]
\draw (0pt, 0pt) circle (7pt);
\draw(7pt, 0pt)--(-7pt, 0pt);
\fill (-4pt,1.2pt)--(-7pt,0pt)--(-4pt,-1.2pt)--cycle;
\draw(0pt, 7pt)--(0pt, -7pt);
\fill (-1.2pt,-4pt)--(0pt,-7pt)--(1.2pt,-4pt)--cycle;
\fill (-4.94975pt, -4.94975pt) circle (1pt);
\draw[font=\tiny] (-10.5pt, 0pt) node {$+$};
\draw[font=\tiny] (0pt, -10.5pt) node {$+$};
\end{tikzpicture}
,~{\text{and}}~\\
f_{2, 2} & :=  \begin{tikzpicture}[baseline=0pt]
\draw (0pt, 0pt) circle (7pt);
\draw(7pt, 0pt)--(-7pt, 0pt);
\fill (-4pt,1.2pt)--(-7pt,0pt)--(-4pt,-1.2pt)--cycle;
\draw(0pt, 7pt)--(0pt, -7pt);
\fill (-1.2pt,-4pt)--(0pt,-7pt)--(1.2pt,-4pt)--cycle;
\fill (4.94975pt, 4.94975pt) circle (1pt);
\draw[font=\tiny] (-10.5pt, 0pt) node {$-$};
\draw[font=\tiny] (0pt, -10.5pt) node {$-$};
\end{tikzpicture}
 + \begin{tikzpicture}[baseline=0pt]
\draw (0pt, 0pt) circle (7pt);
\draw(7pt, 0pt)--(-7pt, 0pt);
\fill (-4pt,1.2pt)--(-7pt,0pt)--(-4pt,-1.2pt)--cycle;
\draw(0pt, 7pt)--(0pt, -7pt);
\fill (-1.2pt,-4pt)--(0pt,-7pt)--(1.2pt,-4pt)--cycle;
\fill (4.94975pt, 4.94975pt) circle (1pt);
\draw[font=\tiny] (-10.5pt, 0pt) node {$-$};
\draw[font=\tiny] (0pt, -10.5pt) node {$+$};
\end{tikzpicture}
 + \begin{tikzpicture}[baseline=0pt]
\draw (0pt, 0pt) circle (7pt);
\draw(7pt, 0pt)--(-7pt, 0pt);
\fill (-4pt,1.2pt)--(-7pt,0pt)--(-4pt,-1.2pt)--cycle;
\draw(0pt, 7pt)--(0pt, -7pt);
\fill (-1.2pt,-4pt)--(0pt,-7pt)--(1.2pt,-4pt)--cycle;
\fill (4.94975pt, 4.94975pt) circle (1pt);
\draw[font=\tiny] (-10.5pt, 0pt) node {$+$};
\draw[font=\tiny] (0pt, -10.5pt) node {$-$};
\end{tikzpicture}
 + \begin{tikzpicture}[baseline=0pt]
\draw (0pt, 0pt) circle (7pt);
\draw(7pt, 0pt)--(-7pt, 0pt);
\fill (-4pt,1.2pt)--(-7pt,0pt)--(-4pt,-1.2pt)--cycle;
\draw(0pt, 7pt)--(0pt, -7pt);
\fill (-1.2pt,-4pt)--(0pt,-7pt)--(1.2pt,-4pt)--cycle;
\fill (4.94975pt, 4.94975pt) circle (1pt);
\draw[font=\tiny] (-10.5pt, 0pt) node {$+$};
\draw[font=\tiny] (0pt, -10.5pt) node {$+$};
\end{tikzpicture}
~.   
\end{align*}
Here, each of the above twenty-one $($two,~resp.$)$ functions is the 
Vassiliev knot invariant of degree three $($two,~resp$)$ or the zero map on 
the set of classical knots.  

In particular, $v_{3, i}$ $($$1 \le i \le 21$$)$ is of degree at most three or zero map
and ${v}_{2, i}$ $($$1 \le i \le 2$$)$ is of degree two as follows:   
\begin{equation}\label{formula2}
{v}_{3,1} = - v_3 \quad (i=1, 2, 5), {v}_{3,7} = v_3 , {v}_{3,i} =0 \quad (i=3, 4, 6, 8 \le i \le 21), {v}_{2,i} = v_2 \quad (i=1,2).
\end{equation}


{\color{black}{As a corollary, for a classical (long) knot $K$, equations ${v}_{3,i} =0~(i=3, 4, 6, 8 \le i \le 21)$ give nontrivial relations among Gauss diagrams. }}
\end{proposition}
{\color{black}{
\begin{remark}
One may worry about the relationship between these $21$   formulas ($v_{3, i}$ ($1 \le i \le 21$)) and the Chmutov-Polyak formula from \cite{CP}.   However, the situation is a bit puzzling.  The $21$ formulas \emph{only} depend on generalized Reidemeister moves and linking number relations (\cite[Theorem~5]{PV}, \cite[Section~4.1]{ostlund}) that are sufficient conditions to give classical knot invariants.   On the other hand,  the Chmutov-Polyak formula \cite{CP} \emph{deeply} depends on the realization of link diagrams on the plane and it is unknown which relations are applied to Gauss diagrams to compare Chmutov-Polyak formula with our invariants. 
\end{remark}
}}

\begin{proof}[Proof of Proposition~\ref{thm2}]
Let ${y}^*_i$ be an arrow diagram as in Notation~\ref{order}.  
{\color{black}{Proposition~\ref{prop2}}} implies that if there exist integers $\alpha_i$ ($1 \le i \le 168$) such that $\sum_{i=1}^{168} \alpha_i \tilde{y}^*_i ({r}^*)$ $=$ $0$ ($\forall {r}^* \in \reduced(\check{R}^{\min}_{10110}(2, 3))$), then $\sum_{i=1}^{168} \alpha_i \tilde{y}^*_i$ is a (long) virtual knot invariant.    
It is elementary to show that $\reduced(\check{O}_{2, 3}(\check{R}^{\min}_{10110}))$ consists of  $404$ elements, where $122$ ($96$, $186$, resp.) relators are concerned with $\check{\ii}$ ($\check{\w}$, $\check{\sss}$, resp.).  
Let $\bf{x}$ $=$ $(y^*_1, y^*_2, \ldots, y^*_{168})$.       
Then, we have such $\alpha_i$ ($1 \le i \le 168$) by solving the linear equation ${\bf{x}} M$ $=$ $\bf{0}$, where $M$ $=$ $(\tilde{y}^*_i ({r}^*_j))_{1\le i \le 168, 1 \le j \le 404}$.     

It is elementary to show that the set of the solutions is given by formulas (\ref{formula2}) in the statement of Proposition~\ref{thm2}.    
Here, recall that the function $v_3$ ($v_2$, resp.), which is the Vassiliev knot invariant of degree three (two, resp.), is unique up to scale.  Suppose that $v_3$ ($v_2$, resp.) takes values $0$ on the unknot, $+1$ on the right trefoil, and $-1$ ($1$,~resp.) on the left trefoil.  
Then, letting $v_{3, i}$ $=$ $\lambda_i$ $v_3$ + $\mu_i$ $v_2$, it is elementary to obtain $\lambda_i$ and $\mu_i$, which give (\ref{formula2}).  
\end{proof}

Let $\bf{v}$ $=$ ${}^t \, (v_{3, 1}, v_{3, 2}, \ldots, v_{3, 21}, v_{2, 1}, v_{2, 2})$ and $\bf{w}$ $=$ ${}^t \, (\tilde
{v}_{3, 1}, \tilde
{v}_{3, 2}, \ldots, \tilde
{v}_{3, 7}, \tilde{v}_{2, 1}, \tilde{v}_{2, 2})$ given by Propositions~\ref{thm1} 
and~\ref{thm2}.  By definition, each Gauss diagram formula $v_{3, j}$ or $v_{2, j}$ ($\tilde{v}_{3, j}$, resp.) is regarded as the vector consisting of $\alpha_i$ obtained by $\sum_{i} \alpha_i \tilde{y}^*_i ({r}^*)$.  Therefore, a vector $\bf{v}$ ($\bf{w}$, resp.) consisting of  functions is identified with the matrix consisting of the coefficients $\{ \alpha_i \}$ of a function $v_{3, j}$ ($v_{2, j}$, $\tilde{v}_{3, j}$, resp.).   We freely use the identification in the next proof.

\section{ {\color{black}{ Proofs of Theorem~\ref{thm3} and Corollary~\ref{PVformula} }} }\label{PrTh3}
\subsection{ {\color{black}{ Proof of Theorem~\ref{thm3} }} }
\begin{proof}
For the fixed arrow diagrams $y_i$ ($1 \le i \le 168$), Proposition~\ref{thm1} (Proposition~\ref{thm2}, resp.) implies $168 \times 9$ ($168 \times 23$, resp.) matrix.  
Let $\bf{v}$ and $\bf{w}$ be the matrix as shown in the statement of Proposition~\ref{thm2}.    
Since $\bf{v}$ has $23$ vectors (corresponding to the rows) are linearly independent, ${\bf{v}} {}^{t}\!{\bf{v}}$ is regular.  Then, letting ${\bf{v}}^+$ $=$ ${}^{t}\!{\bf{v}} ({\bf{v}} {}^{t}\!{\bf{v}})^{-1}$, ${\bf{v}}^+$ satisfies that ${\bf{v}} {\bf{v}}^+$ is the identity matrix.    
Here, let $A$ be a matrix satisfying that $A {\bf{v}}$ $=$ $\bf{w}$.  Then,  
\[
A = A ({\bf{v}} {\bf{v}}^+) = {\bf{w}} {\bf{v}}^+.  
\]
By Propositions~\ref{thm2} and ~\ref{thm1}, since ${\bf{w}}$ and ${\bf{v}}^+$ are fixed, $A$ $=$ ${\bf{w}} {\bf{v}}^+$, which is unique.  
\end{proof}

\subsection{ {\color{black}{Proof of Corollary~\ref{PVformula} }}} 
\begin{proof}
We recall that a long knots topologically equal to a knot with a base point.  Then, for a given by considering the possibilities of the chosen base point on classical knots, 
we see that 
\[ \atrb = \atrba + \atrbb, \]
and 
\[ \aha =  \ahaa+\ahab+\ahac+\ahad+\ahae+\ahaf.
\] 
Then we have $-v_{3, 9}$ $=$ $\Biggl< 2 \atrb +  \aha, \ \cdot \ \Biggl>$.      
\end{proof}

{\color{black}{
\subsection{Proof of Proposition~\ref{unsignedProp} } \label{unsignedProof}
\begin{proof}
First we recall the definition of Type ($\check{\w}$) relator (Definition~\ref{def_relators_arrow}): 
An element $r^*$ of $\mathbb{Z}[\check{G}_{< \infty}]$ is called a \emph{Type $(\check{\w})$ relator} if there exist an oriented Gauss word $STU$ and letters $j$ and $k$ not in $STU$ such that $\sign(j)$ $\neq$ $\sign(k)$ and 
$r^*$ $=$ $[[SjkT\bar{j}\,\bar{k}U]]$ $+$ $[[SjT\bar{j}U]]$ $+$ $[[SkT\bar{k}U]]$.  

Suppose that the length of chord diagram  $[[SjkT\bar{j}\,\bar{k}U]]$ is $n+1$.  Also we suppose that a chord diagram appearing in $\caFsum$ is lower than $n+1$ and the longest length is $n$ among all the term.   
Noting that $\sign(j)$ $\neq$ $\sign(k)$, every chord diagram $x^*$ $=$  $[[SjT\bar{j}U]]$  appearing in $\caFsum$ always satisfies 
\begin{align}\label{wiirel}
\caFsum (
[[SjT\bar{j}U]] + [[SkT\bar{k}U]]) = 0.   
\end{align}
Note also that for any oriented Gauss word $u^*$ and any letter $j$ of $u^*$ of the longest length $n$, there exists oriented Gauss word $STU$ such that $u^*$ $=$ $SjT\bar{j}U$ or $S\bar{j}T j U$. 
Then, by (\ref{wiirel}), we have Lemma~\ref{wiilemma}.  
\begin{lemma}\label{wiilemma}
If $[[SjT\bar{j}U]]$ $\in \{ x^*_i \}_i$ with $\sign(j)$ $=$ $+$ $(resp.~-)$, say $x^*_1 = [[SjT\bar{j}U]]$, then $[[SkT\bar{k}U]]$ $\in \{ x^*_i \}_i$ with $\sign(k)$ $=$ $-$ $(resp.~+)$, say $x^*_2 = [[SkT\bar{k}U]]$.  Then we have  
\[
\alpha_1 = \alpha_2.  
\]
In other word, if 
the length of an arrow diagram $[[SjT\bar{j}U]]$   appearing in a linear combination $\caFsum$ is the longest among all the terms, $\caFsum$ includes 
\[
\alpha [[SjT\bar{j}U]] + \alpha [[SkT\bar{k}U]]
\] where $\sign(j)\neq \sign(k)$ and $\alpha$ is an integer.  
\end{lemma}
Further, for every arrow diagram $[[SjT\bar{j}U]]$ of the longest length appearing in $\caFsum$, we apply Lemma~\ref{wiilemma} to each letter $j$.  Here we recall the notation of unsigned arrow diagrams (Definition~\ref{remarkGauss}).  If an arrow diagram $x^*_i$ of the longest length appears in $\caFsum$, by exchanging order of $\{ x^*_i \}_i$ if necessary, then  the unsigned arrow diagram $x$ ($=$ $\sum_{i=1}^{2^n} x^*_i$) also appears in $\caFsum$.  
\end{proof}
}}

{\color{black}{
\section{Relationships with other  known formulas of long virtual knots}\label{Tip}
In this section, we give a comment on relationships between Proposition~\ref{thm1} with known formulas by Chmutov-Duzhin-Mostovoy \cite{CDM} and by Zohar-Hass-Linial-Nowik \cite{EHLN} of long virtual knots.   These formulas are obtained to correct a Gauss diagram formula \cite[Page~1059]{gpv} of long virtual knots.   

Note that our sign notation  is slightly different from them.  
Using our notation, the Gauss diagram formulas $\langle f_{\text{CDM}} , \cdot \rangle$ of Chmutov-Duzhin-Mostovoy \cite{CDM} and $\langle f_{\text{ZHLN}} , \cdot \rangle$ of  
 and Zohar-Hass-Linial-Nowik \cite{EHLN} are given by 

$f_{\text{CDM}}$ $=$ 
 \input{./399_1sign}
 and 
$f_{\text{ZHLN}}$ $=$
\input{./GPV27sign}
respectively.  Then 
\begin{equation}\label{EqCDM}
f_{\text{CDM}} = - \tilde{f}_{3, 2} - \tilde{f}_{3, 3} + \tilde{f}_{3, 4} +\tilde{f}_{2, 1} +2\begin{tikzpicture}[baseline=0pt]
\draw (0pt, 0pt) circle (7pt);
\draw(7pt, 0pt)--(-7pt, 0pt);
\fill (-4pt,1.2pt)--(-7pt,0pt)--(-4pt,-1.2pt)--cycle;
\draw(0pt, -7pt)--(0pt, 7pt);
\fill (1.2pt,4pt)--(0pt,7pt)--(-1.2pt,4pt)--cycle;
\fill (4.94975pt, -4.94975pt) circle (1pt);
\draw[font=\tiny] (-10.5pt, 0pt) node {\mbox{$-$}};
\draw[font=\tiny] (0pt, 10.5pt) node {\mbox{$+$}};
\end{tikzpicture}
\end{equation}
and 
$f_{\text{ZHLN}}$ $=$ $- \tilde{f}_{3, 2} - \tilde{f}_{3, 3} + \tilde{f}_{3, 4} + \tilde{f}_{2, 2}$.  

Firstly, the formula $f_{\text{ZHLN}}$ of Zohar-Hass-Linial-Nowik \cite{EHLN} gives an invariant.  Secondly, the formula $f_{\text{CDM}}$ of  Chmutov-Duzhin-Mostovoy \cite{CDM} does  \emph{not} give an invariant and still includes a typo.  
 
Assume that $\langle f_{\text{CDM}}, \cdot \rangle$ must be an invariant.  This assumption implies that $\langle \begin{tikzpicture}[baseline=0pt]
\draw (0pt, 0pt) circle (7pt);
\draw(7pt, 0pt)--(-7pt, 0pt);
\fill (-4pt,1.2pt)--(-7pt,0pt)--(-4pt,-1.2pt)--cycle;
\draw(0pt, 7pt)--(0pt, -7pt);
\fill (-1.2pt,-4pt)--(0pt,-7pt)--(1.2pt,-4pt)--cycle;
\fill (4.94975pt, 4.94975pt) circle (1pt);
\draw[font=\tiny] (-10.5pt, 0pt) node {$+$};
\draw[font=\tiny] (0pt, -10.5pt) node {$-$};
\end{tikzpicture}
, \cdot \rangle$ should be an invariant.  However, by the example of Fig.~\ref{counterEg}, it is clear that $\langle \begin{tikzpicture}[baseline=0pt]
\draw (0pt, 0pt) circle (7pt);
\draw(7pt, 0pt)--(-7pt, 0pt);
\fill (-4pt,1.2pt)--(-7pt,0pt)--(-4pt,-1.2pt)--cycle;
\draw(0pt, 7pt)--(0pt, -7pt);
\fill (-1.2pt,-4pt)--(0pt,-7pt)--(1.2pt,-4pt)--cycle;
\fill (4.94975pt, 4.94975pt) circle (1pt);
\draw[font=\tiny] (-10.5pt, 0pt) node {$+$};
\draw[font=\tiny] (0pt, -10.5pt) node {$-$};
\end{tikzpicture}
, \cdot \rangle$ is not invariant, which implies a contradiction.  
We revise it and give a correct one   
\[
\tilde{f}_{\text{CDM}} = - \tilde{f}_{3, 2} - \tilde{f}_{3, 3} + \tilde{f}_{3, 4} +\tilde{f}_{2, 1}
\]
that is an invariant of degree three.  
}}
\section*{{\color{black}{Acknowledgements}}}
The authors NI and YK gratefully acknowledge support from the Simons Center for Geometry and Physics, Stony Brook University at which some of the research for this paper was performed.   
{\color{black}{YK}} is supported by RIKEN iTHEMS Program and RIKEN AIP (Center for Advanced Intelligence Project).  
{\color{black}{The work of NI is partially supported by JSPS  KAKENHI Grant Number JP20K03604.}}   
{\color{black}{The work of YK}} is partially supported by Grant-in-Aid for Young Scientists (B) (No.~16K17586) and Grant-in-Aid for Early-Career Scientists  (No.~20K14322), Japan Society for the Promotion of Science.

\end{document}